\ifdefined\pdfoutput\pdfoutput=1\fi 
\documentclass[12pt,reqno]{amsart}

\usepackage{amssymb,amsthm,amsmath,amstext,amsxtra}
\usepackage[lowtilde]{url}
\usepackage{graphicx}
\usepackage{fullpage}
\usepackage{bm}       
\usepackage{mathtools} 
\usepackage{booktabs}
\usepackage{enumerate}
\usepackage{comment}
\usepackage{colonequals}
\usepackage{subcaption}
\usepackage[all,cmtip]{xy}
\usepackage{tikz-cd}
\usepackage[T1]{fontenc}
\usepackage{nicefrac}

\makeatletter
\g@addto@macro\bfseries{\boldmath} 
\makeatother

\usepackage{microtype}  

\makeatletter
\@namedef{subjclassname@2020}{%
  \textup{2020} Mathematics Subject Classification}
\makeatother

\numberwithin{equation}{section}

\newtheorem{theorem}[equation]{Theorem}
\newtheorem{lemma}[equation]{Lemma}
\newtheorem{corollary}[equation]{Corollary}

\theoremstyle{definition}
\newtheorem{algorithm}[equation]{Algorithm}
\newtheorem{definition}[equation]{Definition}

\newtheorem{remark}[equation]{Remark}

\newcommand{\Z}{\mathbb{Z}}
\newcommand{\Q}{\mathbb{Q}}

\newcommand{\F}{\mathbb{F}}

\let\P\relax
\newcommand{\P}{\mathbb{P}}

\newcommand{\calB}{\mathcal{B}}
\newcommand{\calD}{\mathcal{D}}
\newcommand{\calE}{\mathcal{E}}

\newcommand{\calO}{\mathcal{O}}

\newcommand{\calW}{\mathcal{W}}

\newcommand{\defi}[1]{\textsf{#1}} 	

\newcommand{\ndivides}{\mathrel{\nmid}}

\renewcommand{\leq}{\leqslant}
\renewcommand{\le}{\leqslant}
\renewcommand{\geq}{\geqslant}
\renewcommand{\ge}{\geqslant}

\DeclareMathOperator{\M}{\textsf{M}}

\DeclareMathOperator{\adj}{adj}
\DeclareMathOperator{\divisor}{div}

\newcommand{\id}{\mathrm{id}}

\newcommand{\restr}[2]{{#1\rvert}_{#2}}

\DeclareMathOperator{\tr}{\mathrm{tr}}

\DeclareMathOperator{\disc}{\mathrm{disc}}

\DeclareMathOperator{\Res}{\mathrm{Res}}

\DeclareMathOperator{\rad}{rad}
\DeclareMathOperator{\Proj}{Proj}

\newcommand{\frakm}{\mathfrak{m}}

\newcommand{\Cartier}{\mathcal{C}}

\usepackage{xcolor}
\definecolor{darkred}{HTML}{CC1F1F}
\definecolor{green}{rgb}{.4,.7,.4}
\definecolor{blue}{rgb}{.2,.6,.75}
\definecolor{pastelb}{HTML}{3333FF}
\definecolor{pastelyellow}{rgb}{0.992157, 0.552941, 0.235294}
\definecolor{pastelorange}{rgb}{0.941176, 0.231373, 0.12549}
\definecolor{pastelred}{rgb}{0.741176, 0., 0.14902}
\definecolor{darkbrown}{rgb}{0.25098, 0., 0.0745098}

\newcommand{\mr}[1]{MathSciNet:\,\href{https://mathscinet.ams.org/mathscinet-getitem?mr=#1}{MR#1}}
\newcommand{\arxiv}[1]{arXiv:\,\href{https://arxiv.org/abs/#1}{#1}}
\newcommand{\iacr}[1]{IACR:\,\href{https://eprint.iacr.org/#1}{#1}}
\newcommand{\halinria}[1]{HAL-Inria:\,\href{https://hal.inria.fr/inria-#1/}{#1}}
\newcommand{\hal}[1]{HAL:\,\href{https://hal.archives-ouvertes.fr/hal-#1/}{#1}}

\newenvironment{alignedtikzcd}
{%
\begin{tikzcd}[row sep=0pt,column sep=1pc]%
}%
{%
\end{tikzcd}%
}

\usepackage{xcolor}
\usepackage{color}
\definecolor{mylinkcolor}{rgb}{0.5,0.0,0.0}
\definecolor{myurlcolor}{rgb}{0.0,0.0,0.75}
\usepackage[
	colorlinks, urlcolor=myurlcolor, citecolor=myurlcolor, linkcolor=mylinkcolor,
	pagebackref,
	breaklinks=true
]{hyperref}

\iffalse
\usepackage{csquotes} 
\usepackage[style=ieee-alphabetic,backend=biber, sorting=nyt, doi=false,isbn=false,url=false]{biblatex}
\usepackage{xurl}
\newbibmacro{string+doiurlisbn}[1]{%
  \iffieldundef{doi}{%
    \iffieldundef{url}{%
      \iffieldundef{isbn}{%
        \iffieldundef{issn}{%
          #1%
        }{%
          \href{https://books.google.com/books?vid=ISSN\thefield{issn}}{#1}%
        }%
      }{%
        \href{https://books.google.com/books?vid=ISBN\thefield{isbn}}{#1}%
      }%
    }{%
      \href{\thefield{url}}{#1}%
    }%
  }{%
    \href{https://dx.doi.org/\thefield{doi}}{#1}%
  }%
}
\DeclareFieldFormat{title}{\usebibmacro{string+doiurlisbn}{\mkbibemph{#1}}}
\DeclareFieldFormat[article,incollection,inproceedings]{title}%
    {\usebibmacro{string+doiurlisbn}{\mkbibquote{#1}}}
\fi

\title{Counting points on smooth plane quartics}

\author{Edgar Costa}
\address{Department of Mathematics
  77 Massachusetts Ave.
Cambridge, MA 02139, USA}
\email{edgarc@mit.edu}
\urladdr{\url{https://edgarcosta.org/}}
\thanks{The first and third authors were supported by Simons Foundation grant 550033}

\author{David Harvey}
\address{School of Mathematics and Statistics, University of New South Wales, Sydney NSW
2052, Australia}
\email{d.harvey@unsw.edu.au}
\urladdr{\url{http://web.maths.unsw.edu.au/~davidharvey/}}
\thanks{The second author was supported by the Australian Research Council (grant FT160100219).}

\author{Andrew V. Sutherland}
\address{Department of Mathematics
  77 Massachusetts Ave.
Cambridge, MA 02139, USA}
\email{drew@math.mit.edu}
\urladdr{\url{https://math.mit.edu/~drew/}}

\makeatletter
\hypersetup{pdftitle={\@title}, pdfauthor={\@author}}
\makeatother

\begin{document}

\begin{abstract}
We present efficient algorithms for counting points on a smooth plane quartic curve $X$ modulo a prime $p$.  We address both the case where $X$ is defined over~$\F_p$ and the case where $X$ is defined over $\Q$ and $p$ is a prime of good reduction.
We consider two approaches for computing $\#X(\F_p)$, one which runs in $O(p\log p\log\log p)$ time using $O(\log p)$ space and one which runs in $O(p^{1/2}\log^2\!p)$ time using $O(p^{1/2}\log p)$ space.  Both approaches yield algorithms that are faster in practice than existing methods.  We also present average polynomial-time algorithms for $X/\Q$ that compute $\#X(\F_p)$ for good primes $p\le N$ in $O(N\log^3\! N)$ time using $O(N)$ space. These are the first practical implementations of average polynomial-time algorithms for curves that are not cyclic covers of $\P^1$, which in combination with previous results addresses all curves of genus $g\le 3$.  Our algorithms also compute Cartier--Manin/Hasse--Witt matrices that may be of independent interest.
\end{abstract}
\maketitle

\section{Introduction}

Let $X/\Q$ be a smooth projective curve of genus $g$.  The $L$-function $L(X,s)=\sum_{n\ge 1} a_n n^{-s}$ is a Dirichlet series that is defined by an Euler product $\prod_p L_p(p^{-s})^{-1}$, where $L_p(T)$ is an integer polynomial of degree at most $2g$.  For primes $p$ of good reduction for $X$ the polynomial $L_p(T)$ is the numerator of the \defi{zeta function}
\begin{equation}\label{eq:zeta}
Z_p(T) \coloneqq \exp\left(\sum_{r\ge 1} \#X(\F_{p^r})\frac{T^r}{r}\right) =\frac{L_p(T)}{(1-T)(1-pT)}
\end{equation}
of the reduction of $X$ modulo $p$. The $L$-function $L(X,s)$ and its coefficients $a_n$ are the subject of many outstanding conjectures, including the connection to automorphic forms predicted by the Langlands program, generalizations of the Sato--Tate conjecture, the Lang--Trotter conjecture, and the conjecture of Birch and Swinnerton-Dyer, as well as conjectures about the zeros and special values of $L(X,s)$.  To numerically investigate these conjectures one needs to compute the Dirichlet coefficients $a_n$ for~$n$ up to some bound $N$ that one would like to make as large as possible, and at a minimum, larger than the square root of the conductor of $L(X,s)$ by a significant constant factor.

Since $L(X,s)$ is defined by an Euler product, its coefficients $a_n$ for $n\le N$ are determined by the coefficients $a_{p^e}$ for prime powers $p^e\le N$, almost all of which are Frobenius traces
\[
a_p=p+1-\#X(\F_p)
\]
at primes $p$ of good reduction for $X$.  From a computational perspective, the problem of computing the integers $a_n$ for $n\le N$ is overwhelmingly dominated by the cost of computing Frobenius traces $a_p$ for good primes $p\le N$, equivalently, counting points on $X$ modulo primes $p\le N$ of good reduction, which is the problem we consider here.

There are two existing algorithms that can compute $a_p$ for good primes $p\le N$ in time~$\widetilde O(N)$, which is optimal up to logarithmic factors, since it is quasilinear in the size of the output. The first is Pila's generalization of Schoof's algorithm \cite{Sch85,Pil90,AH01}, which can compute each $a_p$ in time $(\log p)^{O(1)}$, leading to a total time of $N(\log N)^{O(1)}$.  The second is Harvey's average polynomial-time algorithm \cite{Har15}, which can compute $a_p$ for good $p\le N$ in time $O(N\log^3\!N)$.  Neither of these algorithms is meant to be practical for $g>1$, but the second has the distinct advantage that the implicit constant (which increases with $g$) is not in the exponent of the complexity bound.  For $g=1$ both algorithms are practical, but the $\widetilde O(N^{5/4})$ generic group algorithm described in \cite{KS08} is faster for all practical values of $N$.

The case $g=2$ is efficiently addressed by the practical implementation of Harvey's algorithm for hyperelliptic curves given in \cite{HS14} and improved in \cite{HS16}.
Prior work has addressed the case $g=3$ in various special cases, including when $X$ is hyperelliptic, either as a degree-2 cover of $\P^1$ \cite{HS16} or as a degree-2 cover of a pointless conic \cite{HMS16}, and when $X$ is superelliptic, including Picard curves and cyclic 4-covers of $\P^1$ \cite{Sut20}. But the generic case of a smooth plane quartic is not efficiently addressed by any prior work we are aware of.

In this article we consider three practical average polynomial-time algorithms for computing the Frobenius traces $a_p$ of a smooth plane quartic $X/\Q$ at good primes $p\le N$.  As with the average polynomial-time algorithms mentioned above, they all involve the computation of partial products of a sequence of $r\times r$ integer matrices $M_0,\ldots,M_{N-1}$ reduced modulo coprime integers $m_0,\ldots,m_{N-1}$ that include the primes $p\le N$.  This can be accomplished in $O(r^2N\log^3 N)$ time using $O(r^2N\log N)$ space via an accumulating remainder tree, and one can improve the constant factors in the time complexity and reduce the space complexity to $O(r^2N)$ using the accumulating remainder forest described in \cite{HS14,HS16}; see Theorem~\ref{thm:forest} for a precise statement.
As with other average polynomial-time algorithms, one can alternatively use these matrices to count points modulo a particular prime $p$ in two ways: one runs in $O(r^2p\log p\log\log p)$ time using $O(r^2\log p)$ space and the other runs in $O(r^2p^{1/2}\log^2 p)$ time using $O(r^2p^{1/2}\log p)$ space, assuming $r=O(\log p)$.

Our restriction to genus 3 curves effectively fixes $r$, so $r^2$ becomes a constant factor that is hidden in our complexity bounds.  But $r$ takes different values in each of the three algorithms we present, and this has a significant impact on their relative running times.  Constant factors related to the size of the matrix coefficients size also play a role, but they are less significant; see \S \ref{sec:timings} for a detailed discussion and a performance comparison of the three algorithms.

Our algorithms compute the trace of Frobenius $a_p$ by computing the trace of the \defi{Cartier--Manin} matrix $A_p\in \F_p^{3\times 3}$ of the smooth plane quartic $X_p\colon f(x_0,x_1,x_2)=0$ over $\F_p$ given by reducing~$X$ modulo $p$.  The precise definition of $A_p$ is recalled in \S\ref{sec:cartier}, but its entries consist of nine particular coefficients of $f^{p-1}$ and its trace is congruent to $a_p$ modulo $p$, which uniquely determines $a_p$ for $p>144$.  The Cartier--Manin matrix provides additional information about~$X_p$, including the $p$-rank of its Jacobian and the reduction of $L_p(T)$ modulo~$p$, which constrains $L_p(T)$ to~$O(p^{1/2})$ possibilities.  These possibilities can be distinguished in $\widetilde O(p^{1/4})$
time using a probabilistic generic group algorithm working in the Jacobian of~$X$; see \cite{Sut07,KS08,Sut09} for details of the algorithm and see \cite{FOR08} for efficient implementation of the group operation. This does not yield an average polynomial time for computing $L_p(T)$ for good $p\le N$, it would have complexity $\widetilde O(N^{5/4})$, but for the practical range of $N$ this approach is faster in practice than using the average polynomial-time algorithm in \cite{Har15}, which can compute $L_p(T)$ for good $p\le N$ in $O(N \log^3\!N)$ time.

The key differences among the three algorithms we consider lie in the relations that are used to define the matrices $M_i$ and the sizes of these matrices; in particular the value of $r$ may be 66, 28, or 16.  The relations used in \cite{Har15} are based on a deformation approach that in the case of a plane quartic curve $X:=f(x_0,x_1,x_2)=0$ introduces an auxiliary polynomial $g(x_0,x_1,x_2)=x_0^4+x_1^4+x_2^4$ and derives relations between the coefficients that appear in the terms of the binomial expansion of $(f+tg)^{p-1}$, where $t$ is an auxiliary parameter.  These relations yield $66\times 66$ matrices $M_i$.  Rather than using the general algorithm given in \cite{Har15}, which does not not require $X$ to be smooth or even a curve (it can be any hypersurface), one can use these matrices to directly compute the coefficients of $f^{p-1}$ that appear in the Cartier--Manin matrix $A_p$ via \cite[Thm.\,4.1]{Har15}, as we explain in \S\ref{sec:algorithms}.
With appropriate optimizations the resulting algorithm is quite practical and faster than previous approaches, as demonstrated by the timings in Table~\ref{table:Fptimings}.

However, the main focus of this paper is deriving new relations that yield smaller matrices~$M_i$. In contrast to \cite{Har15}, which uses relations that involve coefficients of $m$th-powers of the homogeneous polynomial $F$ that defines $X$, where the parameter $m$ may vary, here we fix~$m$.  This forces us to impose nondegeneracy conditions on~$F$ that are not required in \cite{Har15}, but it yields $28\times 28$ matrices, and the resulting algorithms for computing Cartier--Manin matrices, either for a single prime $p$ or all good $p\le N$ are substantially faster in practice than those that use the $66\times 66$ matrices based on \cite{Har15}.  The relations we obtain are not independent, and we develop tools that allow us to compress them.  This yields $16\times 16$ matrices of full rank with slightly larger coefficients that provides a further substantial improvement in practical running times; see Tables~\ref{table:Fptimings}--\ref{table:sizes}.

Our algorithms for smooth plane quartics are not as fast as those that have been developed for genus 3 curves of a special form, such as hyperelliptic or superelliptic curves; see Table~\ref{table:Qtimings2} for a comparison.  Nevertheless, for general genus 3 curve the algorithms we present substantially extend the practical range of $N$ one may consider.  This played a key role in \cite{FKS21,FKS22} where a preliminary version of our algorithm was used to compute Sato--Tate distributions, and in computing the $L$-functions of the nonhyperelliptic genus 3 curves tabulated in \cite{Sut19}.

We conclude this introduction with an outline of the paper.  After briefly recalling the definition of the Cartier--Manin matrix and some of its properties in Section \ref{sec:cartier}, we devote Sections \ref{sec:setup} and \ref{sec:shifting} to developing the recurrences that determine the matrices $M_i$ used by our algorithms; the main result used to define the $28\times 28$ matrices $M_i$ appears in Lemma~\ref{lemma:construction phi}, and the result that allows us to compress them to $16\times 16$ matrices appears in Lemma~\ref{lemma:compression}. The algorithms themselves are presented in Section~\ref{sec:algorithms}, along with an analysis of their complexity, and Section~\ref{sec:timings} compares the performance of our algorithms to each other and to existing approaches for counting points on smooth plane quartics, as well as to previously developed average polynomial-time algorithms for hyperelliptic and superelliptic genus 3 curves.

\section{The Cartier matrix of a smooth plane curve}\label{sec:cartier}

In this section we recall the definition of the Cartier matrix of a smooth plane curve, following \cite{Sut20}.
Let $k$ be a perfect field of characteristic $p>0$, let $K$ be a function field of transcendence degree one with field of constants $k$,
and let $\Omega_K$ denote its module of differentials, which we identify with its module of Weil differentials via \cite[Def.\,4.17]{Sti09} and \cite[Rm.\,4.3.7]{Sti09}.
Let $x\in K$ be a separating element, so that $K/k(x)$ is a finite separable extension, and let $K^p$ denote the subfield of $p$th powers.  Then $(1,x,\ldots, x^{p-1})$ is a basis for~$K$ as a $K^p$-vector space, and every $z\in K$ has a unique representation of the form
\[
z = z_0^p+z_1^px+\cdots + z_{p-1}^px^{p-1},
\]
with $z_i\in K$. Each rational differential form $\omega = zdx$ can then be written uniquely as
\[
\omega = (z_0^p+z_1^px+\cdots z_{p-1}^px^{p-1})dx.
\]
The (modified) \textsf{Cartier operator} $\Cartier\colon \Omega_K\to\Omega_K$ is then defined by
\[
\Cartier(\omega) \coloneqq z_{p-1}dx.
\]
It maps regular differentials to regular differentials and thus restricts to an operator on the space $\Omega_K(0)\coloneqq\{\omega\in\Omega_K:\omega=0\text{ or } \divisor(\omega)\ge 0\}$, which is a $k$-vector space whose dimension~$g$ is the genus of $K$. See \cite[Ex.\,4.12-17]{Sti09} for these and other standard facts about the Cartier operator.

\begin{definition}
Let $\vec{\omega}\coloneqq (\omega_1,\ldots,\omega_g)$ be a basis for $\Omega_K(0)$ and define $a_{ij}\in k$ via
\[
\Cartier(\omega_j)=\sum_{i=1}^g a_{ij}\omega_i.
\]
The \defi{Cartier--Manin} matrix of $K$ (with respect to $\vec{\omega}$) is the matrix $A\coloneqq [a_{ij}]\in k^{g\times g}$.
\end{definition}

If $X/k$ is a smooth projective curve with function field $K$, we also call $A$ the Cartier--Manin matrix of $X$.  This matrix is closely related to the \defi{Hasse--Witt} matrix $B$ of $X$, which is defined as the matrix of the $p$-power Frobenius operator acting on $H^1(X,\calO_X)$ with respect to some basis.  As explained in \cite{AH19}, the matrices $A$ and $B$ are related by Serre duality, and for a suitable choice of basis one finds that $B = [a_{ij}^p]^{\mathsf{T}}$.
In the case of interest to us $k=\F_p$ is a prime field and the Cartier--Manin and Hasse--Witt matrices are simply transposes, hence have the same rank and characteristic polynomials.  But we shall follow the warning/request of~\cite{AH19} and call $A$ the Cartier--Manin matrix, although one can find examples in the literature where $A$ is called the Hasse--Witt matrix (see \cite{AH19} for a list).

Following St\"ohr--Voloch \cite{SV87} we write $K$ as $k(x)[y]/(F)$, where $x\in X$ is a separating element and $y$ is an integral generator for the finite separable extension $K/k(x)$ with minimal polynomial $F\in k[x][y]$.
We now define the differential operator
\[
\nabla \coloneqq \frac{\partial^{2p-2}}{\partial x^{p-1}\partial y^{p-1}},
\]
which maps $x^{(i+1)p-1}y^{(j+1)p-1}$ to $x^{ip}y^{jp}$ and annihilates monomials not of this form; it thus defines a semilinear map $\nabla\colon K\to K^p$.
Writing $F_y$ for $\frac{\partial}{\partial y} F\in k[x,y]$, for any $h\in K$ we have
\begin{equation}\label{eq:SV1}
\Cartier\left(h\frac{dx}{F_y}\right) = \left(\nabla (F^{p-1}h)\right)^{1/p}\frac{dx}{F_y},
\end{equation}
by \cite[Thm.\,1.1]{SV87}.
If we choose a basis for $\Omega_X(0)$ using regular differentials of the form~$h\frac{dx}{F_y}$, we can compute the action of the Cartier operator on this basis via \eqref{eq:SV1}.  To construct such a basis, we use differentials of the form
\begin{equation}\label{eq:omega}
\omega_{k\ell}\coloneqq x^{k-1}y^{\ell-1}\frac{dx}{F_y}\qquad (k,\ell\ge 1,\ \ k+\ell\le \deg(F)-1).
\end{equation}
Writing $F(x,y)^{p-1}=\sum_{i,j} F^{p-1}_{ij}x^iy^j$ (defining $F^{p-1}_{i,j}\in k$ for all $i,j\in\Z$), for $k,\ell\ge 1$ we have
\begin{equation}\label{eq:SV2}
\nabla\left(\sum_{i,j\ge 0}F^{p-1}_{ij}x^{i+k-1}y^{j+\ell-1}\right) = \sum_{i,j\ge 1} F^{p-1}_{ip-k,\,jp-\ell} x^{(i-1)p}y^{(j-1)p}.
\end{equation}
Now $F^{p-1}_{ip-k,\,jp-\ell}$ is nonzero only when $(i+j)p-(k+\ell)\le (p-1)\deg(F)$, and $k+\ell\le \deg(F)-1$, so we can restrict the sum on the RHS to $i+j\le \deg(F)-1$.
From \eqref{eq:SV1} and \eqref{eq:SV2} we obtain
\begin{equation}\label{eq:SV3}
\Cartier(\omega_{k\ell}) = \sum_{i,j\ge 1} \left(F_{ip-k,\,jp-\ell}^{p-1}\right)^{1/p}\omega_{ij}.
\end{equation}

When $X$ is a smooth plane curve the complete set of $\omega_{ij}$ defined in \eqref{eq:omega} is a basis for~$\Omega_K(0)$ and we can read off the entries of the Cartier--Manin matrix $A$ of $X$ directly from \eqref{eq:SV3}.  Following the convention in \cite{Sut20}, we order our basis $\boldsymbol{\omega}\coloneqq(\omega_{ij})$ for $\Omega_k(0)$ in increasing order by $j$ and then $i$, so that $\boldsymbol{\omega}=(\omega_{11},\omega_{21},\ldots,\omega_{12},\ldots)$, and we view the Cartier--Manin matrix as acting on the column vector $\boldsymbol{\omega}^{\textsf{T}}$, so that we may express \eqref{eq:SV3} as $\mathcal C(\boldsymbol{\omega}^{\textsf{T}})=A\boldsymbol{\omega}^{\textsf{T}}$.

If~ $X\colon f(x_0,x_1,x_2)=0$ is a smooth plane quartic curve with $f(0,1,0)\ne 0$ (an assumption that will hold under non-degeneracy constraints we impose on $X$), then we may write its function field as $k(x)[y]/(F(x,y))$ with $x=x_0/x_2$ and $y=x_1/x_2$ so that its Cartier--Manin matrix with respect to the basis in~\eqref{eq:omega} is
\begin{equation}\label{equation:cartier--main spq}
A = \begin{bmatrix}\vspace{4pt}
f^{p-1}_{p-1,\,p-1,\,2p-2} & f^{p-1}_{2p-1,\,p-1,\,p-2} & f^{p-1}_{p-1,\,2p-1,p-2}\\\vspace{4pt}
f^{p-1}_{p-2,\,p-1,\,2p-1} & f^{p-1}_{2p-2,\,p-1,\,p-1} & f^{p-1}_{p-2,\,2p-1,p-1}\\
f^{p-1}_{p-1,\,p-2,\,2p-1} & f^{p-1}_{2p-1,\,p-2,\,p-1} & f^{p-1}_{p-1,\,2p-2,p-1}
\end{bmatrix},
\end{equation}
where $f_{i,j,k}^{p-1}$ denotes the coefficient of the term $x_0^ix_1^jx_2^k$ in $f(x_0,x_1,x_2)^{p-1}$.

An essential property of the Cartier--Manin matrix is the identity
\begin{equation}\label{eq:CMidentity}
\det(I-TA)\equiv L_p(T)\bmod p,
\end{equation}
where $L_p(T)$ is the numerator of the zeta function of $X$ defined in~\eqref{eq:zeta}; see \cite[Thm.\,3.1]{Katz73} and \cite[Thm.\,1]{Man65}.
In particular, we have $\tr A \equiv a_p \bmod p$, where $a_p$ is the trace of Frobenius.
The Weil bounds imply $|a_p|\le 2g\sqrt{p}$, which allows us to derive $\#X(\F_p)=p+1-a_p$ from $\tr A$ for all $p>16g^2=144$ (for $g=3$).

\begin{remark}
All of our algorithms compute $\#X(\F_p)=p+1-a_p$ by computing the Cartier--Manin matrix $A$ and lifting $\tr A\in \Z/p\Z$ to the unique $a_p\in \Z$ with $|a_p|\le 6\sqrt{p}$ when $p>144$. For $p\le 144$ we are happy to count points na\"ively via \eqref{eq:naive}.
\end{remark}

\section{Setup}\label{sec:setup}

Throughout this section, $R$ denotes one of the rings $\Z$ or $\F_p$.
Many of the results we use hold in greater generality, but we make no attempt to generalize them beyond the cases of interest to us here.

We write $R[x^\pm]$ for the Laurent polynomial ring $R[x_0, x_0^{-1}, \ldots, x_n, x_n^{-1}]$ in $n+1$ variables.
We use multi-index notation: for $v \coloneqq (v_0, \dots, v_n) \in \Z^{n+1}$, we write $x^v$ for the monomial $x_0^{v_0} \cdots x_n^{v_n}$.
For $G \in R[x^\pm]$ we write $G_v$ for the coefficient of $G$ at the monomial~$x^v$.
We also define the \defi{degree} of $v\in \Z^{n+1}$ to be $\deg v\coloneqq \deg x^v=\sum_{i=0}^n v_i$.

For $\ell \in \Z$, we write $R[x^\pm]_\ell$ for the $R$-submodule of $R[x^\pm]$ generated by the monomials of degree $\ell$.
More generally, for any subset $S \subseteq \Z^{n+1}$, we define $R[x^\pm]_S$ to be the $R$-submodule of Laurent polynomials supported on~$S$, consisting of all $G \in R[x^\pm]$ such that $G_v = 0$ for $v \notin S$.
We typically use this notation in the case that $S$ corresponds to a finite set of monomials, all of the same degree.
For $G \in R[x^\pm]$ we define $\restr{G}{S}$, the \defi{restriction of $G$ to $S$},
to be the polynomial $\sum_{v \in S} G_v x^v\in R[x^\pm]_S$.

For any $R$-submodule $M \subseteq R[x^\pm]$, we put $M_\ell \coloneqq M \cap R[x^\pm]_\ell$.
In particular, let $R[x]$ denote the subring $R[x_0, \ldots, x_n]$;
then $R[x]_\ell$ is the submodule of homogeneous polynomials of degree $\ell$,
or the zero submodule if $\ell < 0$.
More generally, if $I$ is a homogeneous ideal of $R[x]$, then $I_\ell$ is the $R$-submodule consisting
of homogeneous polynomials of degree $\ell$ in $I$.
The monomials generating $R[x]_\ell$ are indexed by the set $D_\ell\coloneqq \{v\in \Z_{\ge 0}^{n+1}:\deg v = \ell\}$ of cardinality $\# D_\ell = \dim_R R[x]_\ell = \binom{\ell+n}{n}$ for $\ell \geq 0$,
with $D_\ell = \emptyset$ for $\ell < 0$.

We denote by $K$ the fraction field of $R$, which is either $\Q$ or~$\F_p$.
All of the definitions for $R[x^\pm]$ above may be extended in the obvious way to $K[x^\pm]$.
We write $\P^n _{K} = \Proj K[x]$ for projective $n$-space over $K$.

For the rest of the section we fix a homogeneous polynomial $F \in R[x]_d$ of degree $d\ge 2$.
We always assume that $d \neq 0$ in $R$; in particular, if $R = \F_p$, then we require that $p \ndivides d$.
Our goal is to establish a framework for efficiently computing individual coefficients $F^m_u\coloneqq (F^m)_u$,
for a prescribed integer $m \geq 0$, without computing the entire polynomial $F^m$.
Our strategy will be to observe that $F^m$ satisfies certain partial differential equations (see \eqref{equation:GF}),
which imply various relations between nearby coefficients of $F^m$.

\begin{definition}
For $\ell \in \Z_{\ge 0}$ and $v \in \Z^{n+1}$ we define $D(v, \ell) \coloneqq \{ v - w: w \in D_\ell \} \subseteq \Z^{n+1}$.
The set $D(v,\ell)$ may be thought of as an inverted simplex of size $\ell$ centered at $v$.
\end{definition}

We will study the vectors of coefficients of $\restr{F^m}{D(v,\ell)}$,
for certain small integers $\ell$ and $v \in \Z^{n+1}$ with $\deg v = dm + \ell$.
As we will see, the differential equations lead naturally to relations among these vectors, for fixed $m$, as we vary~$v$.

\begin{remark}
When $n=2$ and $F$ defines a smooth plane curve $X$ in $\P^2_{\F_p}$ of genus $g=\binom{d-1}{2}$, the Cartier--Manin matrix of $X$ consists of~$g^2$ coefficients $F^{p-1}_u$ with $u\in D(v,\ell)$ for $g$ particular choices of $v$ of degree $d(p-1)+\ell$ with $\ell=d-3$. It turns out to be more convenient to use $m=p-2$, as we will eventually want $d(m+1)\ne 0$ in $\F_p$, and to use $v$ of degree $d(p-2)+\ell$ with $\ell = 2d-2$.  For smooth plane quartics we have $n=2$, $d=4$, and $\ell=nd-n=6$, values the reader may find useful to keep in mind.
\end{remark}

Let $I_F$ be the homogeneous ideal $\langle \partial_0 F, \ldots, \partial_n F \rangle$ in $K[x]$,
where $\partial_{i}$ is the degree-preserving differential operator $\partial_i \coloneqq x_i \frac{\partial}{\partial x_i}$.
For $\ell \in \Z$, the $K$-vector space $K[x]_\ell / (I_F)_\ell$ is spanned by the monomials $\{x^\beta : \beta \in D_\ell\}$,
so we may choose a subset $B_\ell \subseteq D_\ell$ such that $\{x^\beta:\beta \in B_\ell\}$  projects to a basis of $K[x]_\ell / (I_F)_\ell$.
For the rest of the discussion, we assume a choice for $B_\ell$ has been fixed for each~$\ell$.
Note that for $\ell < d$ we have $(I_F)_\ell=0$, in which case $B_\ell = D_\ell$.

\begin{definition}
Let $b_\ell \coloneqq \dim_K K[x]_\ell/(I_F)_\ell = \# B_\ell \le \# D_\ell$.
For $v\in \Z^{n+1}$ we define the set $B(v, \ell) \coloneqq \{v - \beta : \beta \in B_\ell\} \subseteq D(v,\ell) \subseteq \Z^{n+1}$.
We also define the $K$-vector spaces $\calD_{v,\ell} \coloneqq {K[x^{\pm}]}_{D(v,\ell)}$
and $\calB_{v,\ell} \coloneqq {K[x^{\pm}]}_{B(v,\ell )} \subseteq \calD_{v, \ell}$.
\end{definition}

We recall the following Hilbert series computation due to Macaulay~\cite{Mac1916}.

\begin{lemma}\label{lemma:macaulay}
  Let $h_0, \ldots, h_n$ be homogeneous polynomials in $K[x]$, of positive degree with no common zeros in $\P^n_K$.
  For $\ell \geq 0$, let
  \begin{equation*}
    \delta_\ell \coloneqq \dim_K K[x]_\ell/\langle h_0, \ldots, h_n \rangle_\ell.
  \end{equation*}
  Then, in $\Z[t]$ we have the identity
  \begin{equation*}
    \sum_{\ell \ge 0} \delta_\ell t^\ell = \prod_{i=0}^n (1 + t + \cdots + t^{\deg h_i - 1}).
  \end{equation*}
\end{lemma}
\begin{proof} See Theorem 58 in \cite[pp.~64--66]{Mac1916}.\end{proof}

Recall that the \textsf{discriminant} $\Delta_d(F)$ of $F\in R[x]_d$ is determined up to sign by the formula
\[
\Delta_d(F) = \pm d^{((-1)^{n+1}-(d-1)^{n+1})/d} \Res_{d-1}\Bigl(\frac{\partial F}{\partial x_0},\ldots,\frac{\partial F}{\partial x_n}\Bigr),
\]
where $\Res_e(h_0,\ldots,h_n)$ is the \textsf{resultant}, the irreducible integer polynomial in the $(n+1)\binom{e+n}{n}$ coefficients of $h_0,\ldots,h_n\in R[x]_e$ that vanishes if and only if $h_0,\ldots,h_n$ have a common zero in $\P_K^n$ and satisfies $\Res_e(x_0^e,\ldots,x_n^e)=1$; see \cite[pp.\,433--435]{GKZ94} for details.

The hypersurface defined by $F\in R[x]_d$ is smooth if and only if
$\partial F / \partial x_0, \ldots, \partial F / \partial x_n$ have no common zeros in $\P_K^n$, that is, if and only if $\Delta_d(F)\ne 0$.
(Note that any common zero of the $\partial F / \partial x_i$ is automatically a zero of $F$ by Euler's identity $d \cdot F = \sum_i \partial_i F$, since $d\ne 0$ in $R$.)
We say that $F$ is \textsf{nondegenerate} if $\partial_0 F, \ldots, \partial_n F$
have no common zeros in $\P_K^n$.  Nondegeneracy of $F$ is equivalent to requiring that the intersection of the hypersurface defined by $F$ with every set of coordinate hyperplanes is smooth (see \cite[Prop.\,4.6]{Bat93}, \cite[Prop.\,1.2]{CV09}); this implies that the hypersurface defined by $F$ is smooth, but it is a stronger condition.  If we let $D_d(S)\coloneqq \{v\in D_d:v_i=0\text{ for } i\in S\}$ and define
\begin{equation}\label{eq:discstar}
\Delta_d^*(F)\coloneqq \prod_{S\subsetneq \{0,\ldots,n\}} \Delta_d\bigl(F|_{D_d(S)}\bigr),
\end{equation}
where the discriminants on the right are taken with respect to the variables not in $S$,
then we see that $F$ is nondegenerate if and only if $\Delta_d^*(F)\ne 0$.

For $n=1$ we have $\Delta_d^*(F)=\pm F_{0,d}F_{d,0}\Delta_d(F) = \pm F_{0,d}F_{d,0}\disc F(t,1)$, where $\disc$ denotes the usual discriminant of a univariate polynomial in $R[t]$. For $n=2$ we have
\[
\Delta_d^*(F)=\pm F_{0,0,d}F_{0,d,0}F_{d,0,0}\disc F(t,1,0)\disc F(t,0,1)\disc F(0,t,1)\Delta_d(F).
\]

Let $H_F(t) \coloneqq \sum_{\ell \ge 0} b_\ell t^\ell\in \Z[t]$ denote the Hilbert series of the quotient ring $K[x]/I_F$.

\begin{corollary}\label{corollary:sum d^n}
 If $F \in {R[x]}_d$ is nondegenerate then
  \begin{equation*}
    H_F(t) \coloneqq \sum_{\ell \ge 0} b_\ell t^\ell = \bigl(1 + t + \cdots + t^{d-1}\bigr)^{n+1},
    \end{equation*}
    and we have $\sum_{\ell \equiv k\bmod d} b_\ell = d^n$ for any integer $k$.
\end{corollary}
\begin{proof}
  The first claim follows from Lemma~\ref{lemma:macaulay}.
  For the second, fix $k\in \Z$ and let $\zeta$ be a primitive $d$th root of unity.
  We have
  \begin{equation*}
      \sum_{i=0}^{d-1} H_F(\zeta^i) \zeta^{-k i}
      = \sum_{i=0}^{d-1} \sum_{\ell \ge 0} b_\ell \zeta^{(\ell-k) i}
      = d\!\!\!\!\! \sum_{\ell\equiv k\bmod d}\!\!\!\!\! b_\ell,
  \end{equation*}
  and also
  \begin{equation*}
      \sum_{i=0}^{d-1} H_F(\zeta^i) \zeta^{-ki}
          = \sum_{i=0}^{d-1} (1 + \zeta^i + \cdots + (\zeta^i)^{d-1})^{n+1} \zeta^{-ki}
          = d^{n +1}.
  \end{equation*}
  Comparing these two expressions yields the desired result.
  \end{proof}

Let $m \geq 0$ and consider the system of differential equations for $G \in K[x^\pm]_{dm}$ given by
\begin{equation}\label{equation:GF}
  \partial_i (FG) = (m+1) (\partial_i F) G, \qquad i = 0, \ldots, n.
\end{equation}
The scalar multiples of $F^m$ are solutions to~\eqref{equation:GF}.
Note that the Euler identity
\begin{equation}\label{equation:Euler}
\sum_{i = 0}^n \partial_{i} (FG) = d (m+1) FG = (m + 1) \sum_{i = 0}^n (\partial_{i} F) G
\end{equation}
implies that one of these $n+1$ equations is redundant,
so for many purposes we may treat it as a system of only $n$ equations.

We now show that~\eqref{equation:GF} defines a system of linear equations on the coefficients of~$G$.
For any $w \in \Z^{n+1}$ of degree $dm+d$, equating coefficients in~\eqref{equation:GF} for the monomial $x^w$ gives rise to the system of linear equations
\begin{equation}\label{equation:GF at u}
  w_i \sum_{t \in D_d} F_t G_{w-t} = (m + 1) \sum_{t \in D_d} t_i F_t G_{w-t}, \qquad i=0,\ldots,n.
\end{equation}
Via \eqref{equation:Euler} we may view this as a system of $n$ equations in $\#D_d$ unknowns $G_u$ for $u\in D(w,d)$.

More generally, for any $\ell \geq d$ and $v \in \Z^{n+1}$ of degree $dm+\ell$ we may consider the system of linear equations involving the coefficients $G_u$ for $u\in D(v, \ell)$,
obtained by including the equations~\eqref{equation:GF at u} for each $w \in D(v,\ell-d)$.
Here we are using the fact that $D(v,\ell)$ is the union of the sets $D(w,d)$
as $w$ ranges over $D(v,\ell-d)$.
Explicitly, these equations are given by
\begin{equation}\label{equation:GF at u big}
   (v_i - s_i) \sum_{t \in D_d} F_t G_{v-s-t} = (m + 1) \sum_{t \in D_d} t_i F_t G_{v-s-t}, \qquad s \in D_{\ell-d},\ i=0,\ldots,n.
\end{equation}
Via \eqref{equation:Euler} we view this as a system of $n\#D_{\ell-d}$ equations in $\#D_\ell$ unknowns~$G_u$ for $u \in D(v,\ell)$.

\begin{definition}
   \label{definition:calE}
   Let $\calE_{v,\ell}$ denote the $K$-vector subspace of $\calD_{v,\ell} = K[x^\pm]_{D(v,\ell)}$
  consisting of those Laurent polynomials $G \in \calD_{v,\ell}$ satisfying the system \eqref{equation:GF at u big}.
\end{definition}

Note that $\calE_{v,\ell}$ is only defined when $\deg v$ is of the form $dm + \ell$ for some $m \geq 0$.
The value of $m$ is implicitly defined by $v$ and $\ell$: we always have $m = (\deg v - \ell)/d$, so a choice of $v$ and~$\ell$ determines a choice of $m$.

Since $F^m$ satisfies the original differential equations \eqref{equation:GF}, we see immediately that
\begin{equation*}
\restr{F^{m}}{D(v,\ell)} \in  \calE_{v, \ell}.
\end{equation*}
We also have the following basic result concerning inclusions of sets of the form $D(v,\ell)$.

\begin{lemma}\label{lemma:calE inclusions}
  Let $\ell, \ell' \geq d$ and let $v, v' \in \Z^{n+1}$ have degrees $d m  + \ell$ and $dm + \ell'$ respectively.
  Assume that $D(v, \ell) \subseteq D(v', \ell')$.
  Then the restriction map $\calD_{v', \ell'} \twoheadrightarrow \calD_{v, \ell}$, $G \mapsto \restr{G}{D(v,\ell)}$, maps $\calE_{v', \ell'}$ into $\calE_{v, \ell}$.
\end{lemma}
\begin{proof}
  The equations defining $\calE_{v, \ell}$ are a subset of those defining $\calE_{v', \ell'}$.
\end{proof}

In the remainder of this section we develop further properties of the vector spaces $\calE_{v, \ell}$.
In particular, we compute their dimension and give explicit bases for certain cases of interest.

\begin{lemma}\label{lemma:compression}
   Let $\ell \geq d$, and let $v \in \Z^{n+1}$ be of degree $dm + \ell$.
  Consider the $K$-linear map
  \begin{equation*}
    \begin{alignedtikzcd}
      \pi_{v, \ell} \colon
      &[-2em]
      \calD_{v, \ell}
      \arrow{r}
      &
      \calB_{v,\ell}
      \oplus
      \calD_{v, \ell - d},
      \\
      &
      G
      \arrow[mapsto]{r}
      &
      \restr{G}{B(v, \ell)}
      +
      \restr{\bigl( FG \bigr)}{D(v, \ell - d)}.
    \end{alignedtikzcd}
  \end{equation*}
  The map $\pi_{v, \ell}$ may be represented by a matrix
  whose entries lie in $R$ and are independent of $v$.

  Moreover, there exists a nonzero constant $\lambda_{\ell} \in R$ and a $K$-linear map
  \begin{equation*}
    \begin{alignedtikzcd}
      \psi_{v,\ell} \colon
      \calB_{v,\ell}
      \oplus
      \calD_{v, \ell - d}
      \arrow{r} & \calD_{v, \ell}
    \end{alignedtikzcd}
  \end{equation*}
  such that the composition
  \[
  \psi_{v,\ell} \circ \pi_{v,\ell} \colon \calD_{v, \ell} \rightarrow \calD_{v, \ell}
  \]
  restricts to scalar multiplication by $(m+1) \lambda_{\ell}$ on $\calE_{v, \ell}$.
  The map $\psi_{v ,\ell}$ may be represented by a matrix whose entries are $R$-linear combinations of~$1,v_0,\ldots,v_n$ and $m$, which we may view as polynomials in $R[v, m] = R[v_0, \ldots, v_n, m]$ of degree at most $1$.
\end{lemma}

Note that when using matrices to represent maps such as $\pi_{v,\ell}$ and $\psi_{v,\ell}$,
we always work with respect to the obvious monomial bases.
For example, the columns of $\pi_{v,\ell}$ are indexed by~$D_\ell$,
and its rows are indexed by the concatenation of $B_\ell$ and $D_{\ell-d}$.
For this purpose we assume that some ordering of the monomials of each degree has been chosen,
such as the lexicographical ordering.

\begin{remark}
   One may think of $\pi_{v,\ell}$ as ``compressing'' a vector of
   length $\# D_\ell$ into a vector of length $\# B_\ell + \# D_{\ell - d}$.
   If the input vector lies in the subspace $\calE_{v,\ell}$,
   i.e., satisfies the appropriate differential equations,
   then no information is lost in the compression,
   and $\psi_{v,\ell}$ ``decompresses'' the result to recover the original vector
   (multiplied by a certain scalar).
\end{remark}

\begin{proof}
  We observe that $\pi_{v,\ell}$ may be represented by a matrix in which
  the rows corresponding to $\calB_{v,\ell}$ have entries in $\{0,1\}$,
  and the entries of the rows corresponding to $\calD_{v, \ell - d}$
  are either zero or of the form $F_u$ for some $u \in D_d$ with $(FG)_{v-w} = \sum_{u \in D_{d}} F_u G_{v-w-u}$ for $w \in D(v, \ell -d)$.  This matrix is the same for every $v \in \Z^{n+1}$ of degree $dm + \ell$.

  We now explain how to construct $\psi_{v,\ell}$.
  Our task is to construct a formula that recovers a polynomial $G \in \calE_{v,\ell}$
  from knowledge of $\restr{G}{B(v,\ell)}$ and $\restr{(FG)}{D(v,\ell-d)}$.

  First, it follows from the definition of $B_\ell$ that for any $u\in D_\ell$ we may write
  \begin{equation}\label{equation:compression}
    \lambda_\ell x^u = \sum_{i = 0}^n h_{u,i} \partial_i F  + \sum_{\beta \in B_\ell} c_{u,\beta} x^\beta,
  \end{equation}
  for some $\lambda_\ell, c_{u,\beta} \in R$ ($\lambda_\ell \neq 0$) and $h_{u,i} \in R[x]_{\ell - d}$.
  (For $u \in B_\ell\subseteq D_\ell$ we may take $h_{u,i} = 0$, $c_{u,u}=\lambda_\ell$, and $c_{u,\beta} = 0$ for $\beta\ne u$.)

  Now suppose that $G\in \calE_{v,\ell}$.
  Multiplying both sides of \eqref{equation:compression} by $(m + 1) G$ and equating coefficients of $x^v$ yields
  \begin{equation*}
     (m +1)\lambda_{\ell} G_{v-u} =
     \sum_{i = 0}^n \sum_{s \in D_{\ell -d}} (m + 1) (h_{u,i})_s \bigl( (\partial_i F) G \bigr)_{v-s}  + (m + 1) \sum_{\beta \in B_\ell}  c_{u,\beta}  G_{v - \beta}
  \end{equation*}
  for each $u \in D_\ell$.
  By assumption $G$ satisfies \eqref{equation:GF at u big},
  so
  \begin{equation}
     (m+1) ((\partial_i F) G)_{v-s} = (\partial_i (FG))_{v-s} = (v_i - s_i) (FG)_{v-s}
  \end{equation}
  for all $s \in D_{\ell-d}$ and $i = 0, \ldots, n$.
  Therefore, for each $u \in D_\ell$,
  \begin{equation}
     (m +1)\lambda_{\ell} G_{v-u}
     = \sum_{i = 0}^n \sum_{s \in D_{\ell - d}} (v_i - s_i) {(h_{u,i})}_s (F G)_{v-s}
     +
     (m + 1) \sum_{\beta \in B_\ell} c_{u,\beta} G_{v-\beta}.
     \label{equation:phicoefficients}
  \end{equation}

  The right hand side of \eqref{equation:phicoefficients} involves the coefficients of $F G$ on $D(v, \ell -d)$
  and the coefficients of $G$ on $B(v, \ell)$,
  so we may use this expression to define~$\psi_{v,\ell}$.
  Explicitly, for $H \in \calB_{v,\ell}$ and $J \in \calD_{v, \ell - d}$  we define $\psi_{v,\ell}(H+J) \in \calD_{v,\ell}$ via
  \begin{equation}\label{equation:phi definition}
    \psi_{v,\ell}(H + J)_{v-u} \coloneqq \sum_{i = 0}^n \sum_{s \in D_{\ell - d}} (v_i - s_i) {(h_{u,i})}_s J_{v-s} + (m + 1) \sum_{\beta \in B_\ell} c_{u,\beta} H_{v-\beta}.
  \end{equation}
  It is clear that the entries of the corresponding matrix are polynomials of degree
  at most $1$ in $v_0, \ldots, v_n, m$ with coefficients in $R$.
  By construction, if $G \in \calE_{v, \ell}$, then \eqref{equation:phicoefficients} implies that
  \[
    \begin{aligned}
      \psi_{v,\ell} (\pi_{v,\ell} (G))_{v - u}
      &=
      \psi_{v,\ell} \Bigl(
        \restr{G}{B(v, \ell)}
        +
        \restr{\bigl(FG \bigr)}{D(v, \ell - d)}
      \Bigr)_{v-u}
      \\
      &=
      \sum_{i = 0}^n \sum_{s \in D_{\ell - d}} (v_i - s_i) {(h_{u,i})}_s (FG)_{v-s} + (m+1)  \sum_{\beta \in B_\ell} c_{u,\beta} G_{v-\beta}
      \\
      &=  (m + 1)\lambda_\ell G_{v-u}
    \end{aligned}
  \]
  for $u \in D_\ell$.
  Thus $\psi_{v,\ell} \circ \pi_{v,\ell}$ restricts to scalar multiplication by $(m+1)\lambda_\ell$ on~$\calE_{v,\ell}$.
\end{proof}

\begin{definition}
We define $\calW_{v, \ell} \coloneqq \calB_{ v,\ell } \oplus \calB_{ v,\ell -d }$.  For $\ell < 2d$ this is the codomain of $\pi_{v,\ell}$ and the domain of $\psi_{v,\ell}$, since $B(v,\ell-d)=D(v,\ell-d)$ for $\ell-d < d$.
\end{definition}

\begin{corollary}\label{corollary:upper bound rank Ev}
  Let $d \le \ell < 2d$ and $v \in \Z^{n+1}$ of degree $dm + \ell$. Assume that $m \neq -1$ in $R$.
  Then
  \begin{equation}\label{equation:upper bound rank Ev}
    \dim_K \calE_{v, \ell} \leq \dim_K \calW_{v,\ell} = b_{\ell} + b_{\ell -d},
  \end{equation}
  and if $F$ is nondegenerate then we have $\dim_K \calE_{v, \ell} \leq d^n$.

  When equality holds in~\eqref{equation:upper bound rank Ev} we may restrict the domain of $\pi_{v,\ell}$ and the codomain of $\psi_{v,\ell}$ to obtain $K$-linear isomorphisms
  \begin{equation*}
    \pi_{v,\ell}^\calE \colon \calE_{v, \ell} \to \calW_{v, \ell}, \qquad
    \psi_{v,\ell}^\calE \colon \calW_{v, \ell} \to \calE_{v, \ell}.
  \end{equation*}
\end{corollary}
\begin{proof}
  As noted above, the hypothesis $\ell < 2d$ ensures that the codomain of $\pi_{v,\ell}$ and domain of $\psi_{v,\ell}$ are both equal to $\calW_{v,\ell}$.
  Let $\lambda_\ell$ be as in Lemma~\ref{lemma:compression}.
  Since $(m + 1)\lambda_{\ell} \neq 0$ in~$R$, Lemma~\ref{lemma:compression} implies that the map $\pi_{v,\ell}$ is injective when restricted to $\calE_{v,\ell}$ (since scalar multiplication by $(m + 1)\lambda_{\ell}$ is injective), and the first inequality follows.
  The equality in \eqref{equation:upper bound rank Ev} is simply the observation that $\dim_K \calW_{v,\ell} = \#B(v,\ell) + \#B(v,\ell-d) = \#B_\ell + \#B_{\ell-d} = b_\ell + b_{\ell-d}$.
  If $F$ is nondegenerate, then by Corollary~\ref{corollary:sum d^n} we have $b_{\ell} + b_{\ell -d} \leq \sum_{\ell' \equiv \ell \bmod d} b_{\ell'} = d^n$.

  Suppose now that equality holds in \eqref{equation:upper bound rank Ev}, so $\dim_K \calE_{v,\ell} = \dim_K \calW_{v,\ell}$.
  Let $\pi_{v,\ell}^\calE \colon \calE_{v,\ell} \to \calW_{v,\ell}$ be the restriction of $\pi_{v,\ell}$ to $\calE_{v,\ell}$.
  As shown in the previous paragraph, $\pi_{v,\ell}^\calE$ is injective, and
  by comparing dimensions we see that it is an isomorphism onto $\calW_{v,\ell}$.
  Then, since $\psi_{v,\ell} \circ \pi_{v,\ell}^\calE \colon \calE_{v,\ell} \to \calD_{v,\ell}$ is injective
  (by Lemma \ref{lemma:compression}) it follows that $\psi_{v,\ell}$ is injective.
  The image of $\psi_{v,\ell}$ contains $\calE_{v,\ell}$
  (again by Lemma \ref{lemma:compression}), and by comparing dimensions we find that its image is equal to $\calE_{v,\ell}$.
  Restricting the codomain of $\psi_{v,\ell}$ then yields the desired isomorphism $\psi_{v,\ell}^\calE \colon \calW_{v,\ell} \to \calE_{v,\ell}$.
\end{proof}

\begin{corollary}\label{corollary:dim = d^2}
  Let $n = 2$, $\ell \in \{2 d -  2, 2d - 1\}$, and $v \in \Z^{n+1}$ of degree $dm + \ell$.
  Then $\dim_K \calE_{v,\ell} \geq d^2$,
  and if $F$ is nondegenerate and $m \neq -1$ in $R$, then $\dim_K \calE_{v,\ell} = b_{\ell} + b_{\ell -d} = d^2$.
\end{corollary}
\begin{proof}
  Recall that $\calE_{v, \ell}$ is defined by a system of $n\#D_{\ell-d}$ equations in $\#D_\ell$ unknowns.
  Its dimension is therefore at least $\#D_\ell - n\#D_{\ell-d} = \binom{\ell + n}{n} - n\binom{\ell-d+n}{n}$,
  which is precisely~$d^2$ for $n = 2$ and $\ell \in \{2 d -  2, 2d - 1\}$, in which case $\dim_K \calE_{v, \ell} \geq d^2$.
  If additionally $F$ is nondegenerate and $m \neq -1$ in $R$, then Corollary~\ref{corollary:upper bound rank Ev} and  Corollary~\ref{corollary:sum d^n} imply that $\dim_K \calE_{v,\ell}\le b_\ell+b_{\ell-d} \le d^2$,
  so we conclude that $\dim_K \calE_{v,\ell} = b_\ell+b_{\ell-d} = d^2$.
\end{proof}

\begin{remark}
Corollaries~\ref{corollary:upper bound rank Ev} and \ref{corollary:dim = d^2} explain why we use $m=p-2$ rather than $m=p-1$ when computing Cartier--Manin matrices: we want $(m+1)\lambda_\ell$ to be nonzero in characteristic~$p$.
\end{remark}

\begin{remark}
  We expect that generalizations of Corollary~\ref{corollary:dim = d^2} for $n > 2$ also hold, that is, $\dim \calE_{v, \ell} = d^n$ for~$F$ nondegenerate and $\ell$ large enough.
  However, a simple dimension count no longer shows that $\pi_{v,\ell}$ is surjective, more is needed.
\end{remark}

\section{Shifting coefficients}\label{sec:shifting}

To simplify the exposition we now specialize to the case $n = 2$.
As in the previous section, $R$ is $\Z$ or $\F_p$, $K$ is its fraction field,
$R[x^\pm]$ is the Laurent polynomial ring in $n+1 = 3$ variables $x_0,x_1,x_2$,
$R[x]$ is the subring $R[x_0, x_1, x_2]$, and we work with a fixed homogeneous polynomial $F\in R[x]_d$ of degree $d>1$ and a positive integer $m$ such that $d(m+1)\ne 0$ in $R$ (we will take $m=p-2$ when $R=\F_p$).
We assume throughout that $F$ is nondegenerate, i.e., that $\Delta_d^*(F) \neq 0$
(see \eqref{eq:discstar} for the definition of $\Delta_d^*(F)$).

Let $e_0, e_1, e_2$ be the standard basis for $\Z^3$.
In this section we consider how to shift a solution to~\eqref{equation:GF at u big} from $D(v, \ell)$ to $D(v + e_i - e_j, \ell)$,
for $\ell=2d-2$ and $v \in \Z^3$ of degree $dm + \ell$,
where $i, j \in \{0, 1, 2\}$ with $i \neq j$.
Our goal is to construct a ``shift'' map
\begin{equation*}
  \tau_{v,i,j} \colon \calD_{v, \ell} \rightarrow \calD_{v + e_i - e_j, \ell},
\end{equation*}
illustrated in the top row of Figure~\ref{fig:solvetriangle}, with two key properties:
\begin{enumerate}
   \item For any $G \in \calD_{v,\ell}$, the coefficients of $G$ and $\tau_{v,i,j}(G)$ should agree on the intersection
   $D(v,\ell) \cap D(v+e_i-e_j,\ell) = D(v-e_j,\ell-1)$, up to multiplication by a known nonzero scalar.
   The region $D(v-e_j,\ell-1)$ is indicated by the dotted lines in Figure~\ref{fig:solvetriangle}.
   \item $\tau_{v,i,j}$ should restrict to a map
   \begin{equation*}
   \tau_{v,i,j}^\calE \colon \calE_{v, \ell} \rightarrow \calE_{v + e_i - e_j, \ell},
   \end{equation*}
   i.e., if $G \in \calD_{v,\ell}$ satisfies the differential equations on $D(v,\ell)$,
   then the shifted polynomial $\tau_{v,i,j}(G)$ satisfies the equations on $D(v+e_i-e_j,\ell)$.
\end{enumerate}
It will be convenient to reformulate the first condition as follows.
For any $\ell' \geq 1$, $w \in \Z^3$ and $k \in \{0,1,2\}$ let
\begin{equation*}
   P_{w,\ell',k} \colon \calD_{w, \ell'} \twoheadrightarrow \calD_{w - e_k, \ell' - 1}
\end{equation*}
denote the restriction map $G\mapsto \restr{G}{D(w - e_k,\ell' - 1)}$ induced by the inclusion $D(w - e_k, \ell' - 1) \subseteq D(w, \ell')$.
Then condition (1) is equivalent to requiring $P_{v+e_i-e_j,\ell,i} \circ \tau_{v,i,j} \colon \calD_{v,\ell} \to \calD_{v-e_j,\ell-1}$
to be a nonzero scalar multiple of $P_{v,\ell,j} \colon \calD_{v,\ell} \to \calD_{v-e_j,\ell-1}$.

\begin{remark}
Later we will apply this framework to $G = \restr{F^m}{D(v,\ell)}$.
It is easy to compute $\restr{F^m}{D(v,\ell)}$ when $v$ is near $d m e_k$,
i.e., at the corners of the simplex.
By repeatedly applying the $\tau_{v,i,j}$ maps,
we may shift this solution to obtain $\restr{F^m}{D(v,\ell)}$ for a given target value of $v$.
For certain carefully chosen $v$, the components of these vectors will in turn yield the entries of the Cartier--Manin matrix
of the smooth plane quartic defined by~$F$, when $d=4$, $\ell=6$ and $m=p-2$.
These shifts are illustrated in Figure~\ref{fig:triangle}.
\end{remark}

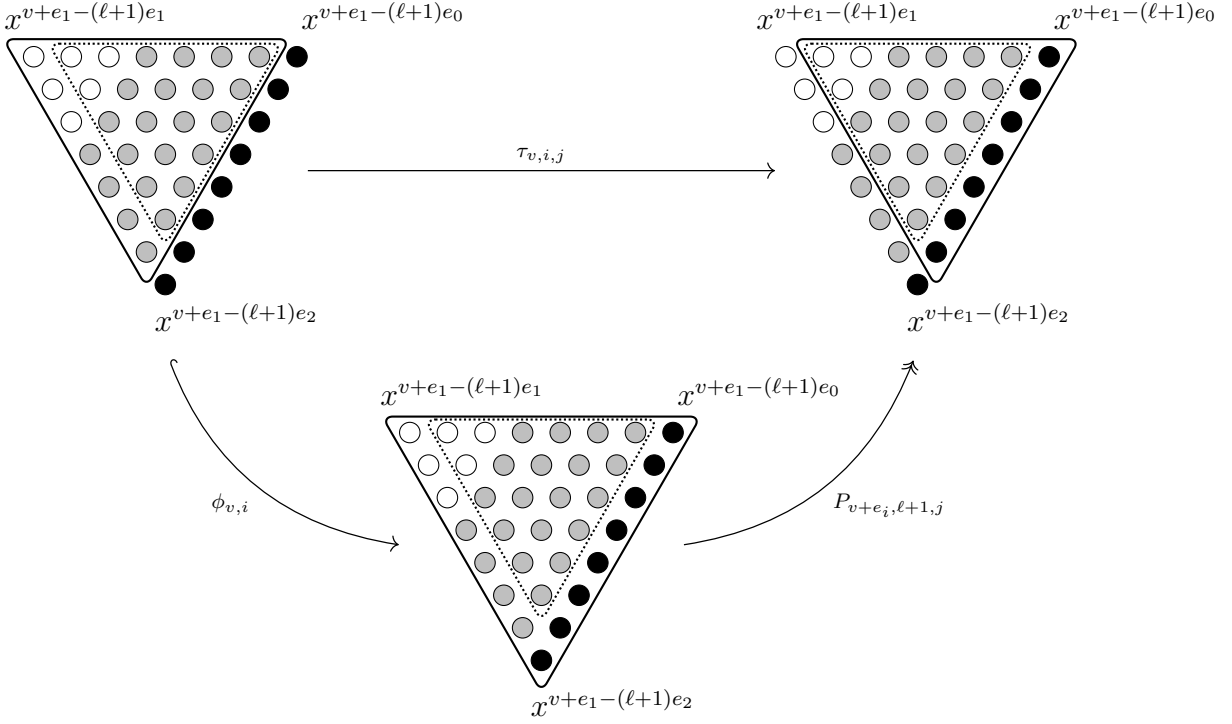
\begin{figure}[htp]
\begin{center}
\begin{tikzpicture}[scale=0.50]
\draw[black,fill=white] (0.000, 6.928) circle (1.5ex);
\draw[black,fill=white] (0.500, 6.062) circle (1.5ex);
\draw[black,fill=white] (1.000, 5.196) circle (1.5ex);
\draw[black,fill=white] (1.500, 6.062) circle (1.5ex);
\draw[black,fill=white] (2.000, 6.928) circle (1.5ex);
\draw[black,fill=white] (1.000, 6.928) circle (1.5ex);
\draw[black,fill=lightgray] (1.500, 4.330) circle (1.5ex);
\draw[black,fill=lightgray] (2.000, 5.196) circle (1.5ex);
\draw[black,fill=lightgray] (5.000, 6.928) circle (1.5ex);
\draw[black,fill=lightgray] (2.500, 6.062) circle (1.5ex);
\draw[black,fill=lightgray] (5.500, 6.062) circle (1.5ex);
\draw[black,fill=lightgray] (3.000, 1.732) circle (1.5ex);
\draw[black,fill=lightgray] (6.000, 6.928) circle (1.5ex);
\draw[black,fill=lightgray] (3.000, 3.464) circle (1.5ex);
\draw[black,fill=lightgray] (3.000, 6.928) circle (1.5ex);
\draw[black,fill=lightgray] (3.500, 2.598) circle (1.5ex);
\draw[black,fill=lightgray] (3.500, 4.330) circle (1.5ex);
\draw[black,fill=lightgray] (2.500, 2.598) circle (1.5ex);
\draw[black,fill=lightgray] (4.000, 3.464) circle (1.5ex);
\draw[black,fill=lightgray] (2.500, 4.330) circle (1.5ex);
\draw[black,fill=lightgray] (4.500, 4.330) circle (1.5ex);
\draw[black,fill=lightgray] (3.000, 5.196) circle (1.5ex);
\draw[black,fill=lightgray] (4.000, 5.196) circle (1.5ex);
\draw[black,fill=lightgray] (2.000, 3.464) circle (1.5ex);
\draw[black,fill=lightgray] (5.000, 5.196) circle (1.5ex);
\draw[black,fill=lightgray] (4.500, 6.062) circle (1.5ex);
\draw[black,fill=lightgray] (4.000, 6.928) circle (1.5ex);
\draw[black,fill=lightgray] (3.500, 6.062) circle (1.5ex);
\draw[black,fill=black] (6.000, 5.196) circle (1.5ex);
\draw[black,fill=black] (4.500, 2.598) circle (1.5ex);
\draw[black,fill=black] (3.500, 0.866) circle (1.5ex);
\draw[black,fill=black] (4.000, 1.732) circle (1.5ex);
\draw[black,fill=black] (7.000, 6.928) circle (1.5ex);
\draw[black,fill=black] (5.000, 3.464) circle (1.5ex);
\draw[black,fill=black] (5.500, 4.330) circle (1.5ex);
\draw[black,fill=black] (6.500, 6.062) circle (1.5ex);
\node at (1.400, 8.028) {$x^{v + e_1 -(\ell + 1)e_1}$};
\node at (9.300, 8.028) {$x^{v  + e_1 -(\ell + 1)e_0}$};
\node at (5.400, 0.000) {$x^{v + e_1 -(\ell + 1)e_2}$};
\node (LR) at (7.00000000000000, 3.89711431702997) {};
\node (LB) at (3.50000000000000, -0.866025403784439) {};
\draw[rounded corners, thick] (-0.810000000000000, 7.39585694831911) -- (6.81000000000000, 7.39585694831911) -- (3.00000000000000, 0.796743371481684) -- cycle;
\draw[rounded corners, densely dotted, thick] (0.400000000000000, 7.27461339178928) -- (6.60000000000000, 7.27461339178928) -- (3.50000000000000, 1.90525588832577) -- cycle;\draw[black,fill=white] (20.000, 6.928) circle (1.5ex);
\draw[black,fill=white] (20.500, 6.062) circle (1.5ex);
\draw[black,fill=white] (21.000, 5.196) circle (1.5ex);
\draw[black,fill=white] (21.500, 6.062) circle (1.5ex);
\draw[black,fill=white] (22.000, 6.928) circle (1.5ex);
\draw[black,fill=white] (21.000, 6.928) circle (1.5ex);
\draw[black,fill=lightgray] (21.500, 4.330) circle (1.5ex);
\draw[black,fill=lightgray] (22.000, 5.196) circle (1.5ex);
\draw[black,fill=lightgray] (25.000, 6.928) circle (1.5ex);
\draw[black,fill=lightgray] (22.500, 6.062) circle (1.5ex);
\draw[black,fill=lightgray] (25.500, 6.062) circle (1.5ex);
\draw[black,fill=lightgray] (23.000, 1.732) circle (1.5ex);
\draw[black,fill=lightgray] (26.000, 6.928) circle (1.5ex);
\draw[black,fill=lightgray] (23.000, 3.464) circle (1.5ex);
\draw[black,fill=lightgray] (23.000, 6.928) circle (1.5ex);
\draw[black,fill=lightgray] (23.500, 2.598) circle (1.5ex);
\draw[black,fill=lightgray] (23.500, 4.330) circle (1.5ex);
\draw[black,fill=lightgray] (22.500, 2.598) circle (1.5ex);
\draw[black,fill=lightgray] (24.000, 3.464) circle (1.5ex);
\draw[black,fill=lightgray] (22.500, 4.330) circle (1.5ex);
\draw[black,fill=lightgray] (24.500, 4.330) circle (1.5ex);
\draw[black,fill=lightgray] (23.000, 5.196) circle (1.5ex);
\draw[black,fill=lightgray] (24.000, 5.196) circle (1.5ex);
\draw[black,fill=lightgray] (22.000, 3.464) circle (1.5ex);
\draw[black,fill=lightgray] (25.000, 5.196) circle (1.5ex);
\draw[black,fill=lightgray] (24.500, 6.062) circle (1.5ex);
\draw[black,fill=lightgray] (24.000, 6.928) circle (1.5ex);
\draw[black,fill=lightgray] (23.500, 6.062) circle (1.5ex);
\draw[black,fill=black] (26.000, 5.196) circle (1.5ex);
\draw[black,fill=black] (24.500, 2.598) circle (1.5ex);
\draw[black,fill=black] (23.500, 0.866) circle (1.5ex);
\draw[black,fill=black] (24.000, 1.732) circle (1.5ex);
\draw[black,fill=black] (27.000, 6.928) circle (1.5ex);
\draw[black,fill=black] (25.000, 3.464) circle (1.5ex);
\draw[black,fill=black] (25.500, 4.330) circle (1.5ex);
\draw[black,fill=black] (26.500, 6.062) circle (1.5ex);
\node at (21.400, 8.028) {$x^{v + e_1 -(\ell + 1)e_1}$};
\node at (29.300, 8.028) {$x^{v  + e_1 -(\ell + 1)e_0}$};
\node at (25.400, 0.000) {$x^{v + e_1 -(\ell + 1)e_2}$};
\node (RL) at (20.0000000000000, 3.89711431702997) {};
\node (RB) at (23.5000000000000, -0.866025403784439) {};
\draw[rounded corners, thick] (20.1900000000000, 7.39585694831911) -- (27.8100000000000, 7.39585694831911) -- (24.0000000000000, 0.796743371481684) -- cycle;
\draw[rounded corners, densely dotted, thick] (20.4000000000000, 7.27461339178928) -- (26.6000000000000, 7.27461339178928) -- (23.5000000000000, 1.90525588832577) -- cycle;\draw[black,fill=white] (10.000, -3.072) circle (1.5ex);
\draw[black,fill=white] (10.500, -3.938) circle (1.5ex);
\draw[black,fill=white] (11.000, -4.804) circle (1.5ex);
\draw[black,fill=white] (11.500, -3.938) circle (1.5ex);
\draw[black,fill=white] (12.000, -3.072) circle (1.5ex);
\draw[black,fill=white] (11.000, -3.072) circle (1.5ex);
\draw[black,fill=lightgray] (11.500, -5.670) circle (1.5ex);
\draw[black,fill=lightgray] (12.000, -4.804) circle (1.5ex);
\draw[black,fill=lightgray] (15.000, -3.072) circle (1.5ex);
\draw[black,fill=lightgray] (12.500, -3.938) circle (1.5ex);
\draw[black,fill=lightgray] (15.500, -3.938) circle (1.5ex);
\draw[black,fill=lightgray] (13.000, -8.268) circle (1.5ex);
\draw[black,fill=lightgray] (16.000, -3.072) circle (1.5ex);
\draw[black,fill=lightgray] (13.000, -6.536) circle (1.5ex);
\draw[black,fill=lightgray] (13.000, -3.072) circle (1.5ex);
\draw[black,fill=lightgray] (13.500, -7.402) circle (1.5ex);
\draw[black,fill=lightgray] (13.500, -5.670) circle (1.5ex);
\draw[black,fill=lightgray] (12.500, -7.402) circle (1.5ex);
\draw[black,fill=lightgray] (14.000, -6.536) circle (1.5ex);
\draw[black,fill=lightgray] (12.500, -5.670) circle (1.5ex);
\draw[black,fill=lightgray] (14.500, -5.670) circle (1.5ex);
\draw[black,fill=lightgray] (13.000, -4.804) circle (1.5ex);
\draw[black,fill=lightgray] (14.000, -4.804) circle (1.5ex);
\draw[black,fill=lightgray] (12.000, -6.536) circle (1.5ex);
\draw[black,fill=lightgray] (15.000, -4.804) circle (1.5ex);
\draw[black,fill=lightgray] (14.500, -3.938) circle (1.5ex);
\draw[black,fill=lightgray] (14.000, -3.072) circle (1.5ex);
\draw[black,fill=lightgray] (13.500, -3.938) circle (1.5ex);
\draw[black,fill=black] (16.000, -4.804) circle (1.5ex);
\draw[black,fill=black] (14.500, -7.402) circle (1.5ex);
\draw[black,fill=black] (13.500, -9.134) circle (1.5ex);
\draw[black,fill=black] (14.000, -8.268) circle (1.5ex);
\draw[black,fill=black] (17.000, -3.072) circle (1.5ex);
\draw[black,fill=black] (15.000, -6.536) circle (1.5ex);
\draw[black,fill=black] (15.500, -5.670) circle (1.5ex);
\draw[black,fill=black] (16.500, -3.938) circle (1.5ex);
\node at (11.400, -1.972) {$x^{v + e_1 -(\ell + 1)e_1}$};
\node at (19.300, -1.972) {$x^{v  + e_1 -(\ell + 1)e_0}$};
\node at (15.400, -10.260) {$x^{v + e_1 -(\ell + 1)e_2}$};
\node (BL) at (10.0000000000000, -6.10288568297003) {};
\node (BR) at (17.0000000000000, -6.10288568297003) {};
\draw[rounded corners, thick] (9.26500000000000, -2.64744432187012) -- (17.7350000000000, -2.64744432187012) -- (13.5000000000000, -9.98267949192431) -- cycle;
\draw[rounded corners, densely dotted, thick] (10.4000000000000, -2.72538660821072) -- (16.6000000000000, -2.72538660821072) -- (13.5000000000000, -8.09474411167423) -- cycle;
\path[commutative diagrams/.cd, every arrow, every label]
(LR) edge node {$\tau_{v,i,j}$} (RL);
\path[commutative diagrams/.cd, every arrow, every label]
   (LB) edge[commutative diagrams/hook, bend right, swap] node {$\phi_{v,i}$} (BL);
\path[commutative diagrams/.cd, every arrow, every label]
   (BR) edge[commutative diagrams/two heads, bend right, swap] node {$P_{v + e_i,\ell+1,j}$} (RB);
\end{tikzpicture}
\end{center}
\caption{Illustration of the maps $\phi_{v,i}$ and $\tau_{v,i,j}$ for $d=4$, $\ell=6$, $i = 1$, $j=0$.
The common domain $\calD(v, \ell)$ of $\tau_{v,i,j}$ and $\phi_{v,i}$ is represented by the white and gray dots enclosed in the upper left triangle (the dots represent a monomial basis).
The codomain $\calD(v+e_i-e_j,\ell)$ of $\tau_{v,i,j}$ is represented by the subset of white, gray, and black dots enclosed in the upper right triangle, and the codomain $\calD(v + e_i, \ell+1)$ of $\phi_{v,i}$ is the entire bottom triangle, which contains both $\calD(v,\ell)$ and $\calD(v+e_i-e_j,\ell)$.
As shown in the proof of Lemma~\ref{lemma:construction phi}, the coordinates in the codomain of $\phi_{v,i}$ represented by the black dots are determined by the coordinates represented by the gray dots.
}
\label{fig:solvetriangle}
\end{figure}

 By composing $\phi_{v,i}$ with the projection $P_{v + e_i,\ell+1,j} \colon \calD_{v + e_i, \ell + 1} \twoheadrightarrow \calD_{v + e_i - e_j, \ell}$ we obtain the desired map $\tau_{v,i,j}$, as shown in the following commutative diagram:
 \begin{equation}
   \label{equation:calD cd}
   \begin{tikzcd}[row sep=1pc,column sep=1pc]
     \calD_{v, \ell}
     \arrow[rr, "\tau_{v,i,j}"]
     \arrow[rd, bend right=15, swap, "\phi_{v,i}"]
     &&
     \calD_{v + e_i - e_j, \ell}\,.
     \\
     &
     \calD_{v + e_i, \ell + 1}
     \arrow[ru, bend right=15, twoheadrightarrow, swap, "P_{v + e_i,\ell+1,j}"]
   \end{tikzcd}
 \end{equation}
 See Figure~\ref{fig:solvetriangle} for an illustration of this diagram in the case $d=4$.

 \begin{equation}
   \begin{tikzcd}[row sep=1pc,column sep=1pc]
     \calE_{v, \ell}
     \arrow[rr, "\tau_{v,i,j}^\calE"]
     \arrow[rd, bend right=15, swap, "\phi_{v,i}^\calE"]
     &&
     \calE_{v + e_i - e_j, \ell}\,.
     \\
     &
     \calE_{v + e_i, \ell + 1}
     \arrow[ru, bend right=15, swap, "P_{v + e_i,j}^\calE"]
   \end{tikzcd}
 \end{equation}

The first step in defining $\tau_{v,i,j}$ is to construct an ``extension'' map
$\phi_{v,i} \colon \calD_{v, \ell} \to \calD_{v + e_i, \ell + 1}$
that extends $G$ from $D(v,\ell)$ to the larger set $D(v+e_i,\ell+1)$.
This is carried out in Lemma~\ref{lemma:construction phi} below.
The idea is to explicitly solve the system~\eqref{equation:GF at u big} for the unknown coefficients of $\phi_{v,i}(G)$,
i.e., for the monomials in $D(v + e_i, \ell + 1) \setminus D(v, \ell)$.
These are shown as the black dots in Figure~\ref{fig:solvetriangle}.

We remind the reader that $n=2$, $d>1$, $\ell = 2d - 2$, $d(m+1)$ is nonzero in $R$,
and $\Delta_d^*(F) \neq 0$.
In particular, $\Delta_d^*(F_{x_i=0}) \neq 0$,
since the latter is a factor of $\Delta_d^*(F)$; see \eqref{eq:discstar}.

\begin{lemma}\label{lemma:construction phi}
  Let $v \in \Z^3$ be of degree $dm + \ell$, let $i \in \{0, 1, 2\}$,
  and let $\theta_i \coloneqq \pm\Delta_d^*(F_{x_i=0}) \neq 0$.
  There exists a $K$-linear map
  \begin{equation*}
    \phi_{v,i} \colon \calD_{v, \ell} \to \calD_{v + e_i, \ell + 1}
  \end{equation*}
  such that $P_{v + e_i,\ell+1,i} \circ \phi_{v,i}  = (v_i + 1) \theta_i \cdot \id_{\calD_{v, \ell}}$,
  and such that if $v_i+1 \neq 0$ in $R$ then
  $\phi_{v,i}(\calE_{v, \ell}) \subseteq \calE_{v + e_i, \ell + 1}$.

  The map $\phi_{v, i}$ may be represented by a $\binom{2d+1}{2} \times \binom{2d}{2}$ matrix whose entries are $R$-linear combinations of~$1, v_0, v_1, v_2$ and $m$, which may be viewed as linear polynomials in~$R[v, m] = R[v_0,v_1,v_2,m]$.
\end{lemma}

\begin{remark}
  The sign of $\theta_i$ is not canonically determined;
  it depends on choices made during the following proof (such as the choice of $j$ and $k$).
  An explicit formula for $\theta_i$, as the determinant of a certain Sylvester matrix,
  is given in \eqref{equation:theta-defn}.
\end{remark}

\begin{proof}
  We are given as input $G \in \calD_{v,\ell}$, and we wish to extend it to some $\widetilde{G} \in \calD_{v+e_i,\ell+1}$.
  We first set $\widetilde{G}_w \coloneqq G_w$ for $w \in D(v,\ell)$.
  Let
  \begin{equation*}
     S \coloneqq D(v + e_i, \ell + 1) \setminus D(v, \ell).
  \end{equation*}
  Our task is to show how to define the missing coefficients $\widetilde{G}_w$ for $w \in S$ in such a way that $\widetilde{G} \in \calE_{v+e_i,\ell+1}$ whenever $G \in \calE_{v,\ell}$.
  These $2d$ coefficients are indicated by the black dots in Figure~\ref{fig:solvetriangle}.
  We can alternatively write $S$ as
  \begin{equation*}
     S = \big\{ (v + e_i) - ( c e_j + (2d-1-c) e_k ) : 0 \leq c \leq 2d-1 \big\}
  \end{equation*}
  where $j$ and $k$ are chosen so that $\{j,k\} = \{0,1,2\}\setminus \{i\}$.
  
  According to \eqref{equation:GF at u big}, $\widetilde{G}$ lies in $\calE_{v+e_i,\ell+1}$ if and only if
  \begin{equation}
     \label{equation:tildeG}
    ((v + e_i)_h - s_h) \sum_{t \in D_d} F_t \widetilde G_{v+e_i-s-t}
    = (m + 1) \sum_{t \in D_d} t_h F_t \widetilde G_{v+e_i-s-t}
  \end{equation}
  for all $s \in D_{\ell+1-d}$ and $h=0,1,2$.
  Consider the subset of equations in \eqref{equation:tildeG} corresponding to those $s$ with $s_i \geq 1$, i.e., for those $s = s'+e_i$ with $s' \in D_{\ell-d}$:
  \begin{equation*}
    (v_h - s'_h) \sum_{t \in D_d} F_t \widetilde G_{v-s'-t}
    = (m + 1) \sum_{t \in D_d} t_h F_t \widetilde G_{v-s'-t}, \qquad s' \in D_{\ell-d}, \ h = 0,1,2.
  \end{equation*}
  These equations only involve $\widetilde G_w$ for $w \in D(v,\ell)$,
  and in fact are exactly the equations defining $\calE_{v,\ell}$.
  If $G \in \calE_{v,\ell}$, then $\widetilde G$ automatically satisfies these equations,
  since we already arranged that $\widetilde G_w = G_w$ for $w \in D(v,\ell)$.
  The remaining equations correspond to those $s \in D_{\ell+1-d} = D_{d-1}$ for which $s_i = 0$,
  i.e., to $s \in E$ where
  \begin{equation*}
     E \coloneqq \{ a e_j + (d-1-a) e_k : 0 \leq a \leq d-1 \}.
  \end{equation*}
  Consequently, for $\widetilde G$ to lie in $\calE_{v+e_i,\ell+1}$,
  it suffices to choose $\widetilde{G}_w$ for $w \in S$ so that \eqref{equation:tildeG} holds for all $s \in E$ and $h = 0, 1, 2$.
  Moreover, we recall that one value of $h$ is redundant, thanks to the Euler identity \eqref{equation:Euler}.
  Taking $h = i$ and $h = j$, this system of $2|E| = 2d$ equations is given explicitly by
  \begin{equation}
    \label{equation:system-v1}
    \begin{aligned}
    (v_i + 1) \sum_{t \in D_d} F_t \widetilde G_{v+e_i-s-t} & = (m + 1) \sum_{t \in D_d} t_i F_t \widetilde G_{v+e_i-s-t}, \qquad s \in E, \\
    (v_j - s_j) \sum_{t \in D_d} F_t \widetilde G_{v+e_i-s-t} & = (m + 1) \sum_{t \in D_d} t_j F_t \widetilde G_{v+e_i-s-t}, \qquad s \in E.
    \end{aligned}
  \end{equation}

  Let us manipulate these equations to put them into a more useful form.
  For each~$s$, multiply the second equation by $v_i + 1$, subtract $v_j - s_j$ times the first equation,
  and divide by $m+1 \neq 0$, to obtain the system
  \begin{equation}
    \label{equation:system-v2}
    \begin{aligned}
    (v_i + 1) \sum_{t \in D_d} F_t \widetilde G_{v+e_i-s-t} & = (m + 1) \sum_{t \in D_d} t_i F_t \widetilde G_{v+e_i-s-t}, & s \in E, \\
    \sum_{t \in D_d} \big((v_i+1) t_j - (v_j - s_j)t_i \big) F_t \widetilde G_{v+e_i-s-t} & = 0, & s \in E.
    \end{aligned}
  \end{equation}
  The system \eqref{equation:system-v2} is equivalent to \eqref{equation:system-v1}, provided that $v_i + 1 \neq 0$.
  Now we rearrange so that the terms with $t_i = 0$ appear on the left hand side:
  \begin{equation}
    \label{equation:system-v3}
    \begin{aligned}
       (v_i + 1) \sum_{\substack{t \in D_d \\t_i=0}} F_t \widetilde G_{v+e_i-s-t} & = \sum_{\substack{t \in D_d \\t_i\neq0}} \big((m+1) t_i - (v_i + 1)\big) F_t \widetilde G_{v+e_i-s-t}, & s \in E, \\
    (v_i+1) \sum_{\substack{t \in D_d \\t_i=0}} t_j F_t \widetilde G_{v+e_i-s-t} & = \sum_{\substack{t \in D_d \\t_i\neq0}} \big((v_j - s_j)t_i - (v_i+1) t_j \big) F_t \widetilde G_{v+e_i-s-t}, & s \in E.
    \end{aligned}
  \end{equation}
  
  We may rewrite the system \eqref{equation:system-v3} in matrix form as follows.
  \begin{itemize}
    \item
    The coefficients $\widetilde{G}_w$ on the left hand side are exactly the unknowns of interest:
    writing $t = b e_j + (d - b) e_k$ for $0 \leq b \leq d$ and $s = a e_j + (d - 1 - a)e_k$ for $0 \leq a \leq d-1$,
    we see that $w = v+e_i-s-t = (v+e_i) - ce_j - (2d-1-c)e_k \in S$ for $c = a + b$.
    Let $y \in K^{2d}$ represent this vector of unknowns, with $y_c = \widetilde{G}_{v+e_i-ce_j-(2d-1-c)e_k}$
    for $0 \leq c \leq 2d-1$.
    
    \item
    The coefficients $\widetilde{G}_w$ on the right hand side are shown as the gray dots in Figure~\ref{fig:solvetriangle}.
    These coefficients are already known, i.e., all such $w$ lie in $D(v,\ell)$, so that $\widetilde{G}_w = G_w$.
    Indeed, if $t = t' + e_i$ for $t' \in D_{d-1}$, then $w = v+e_i-s-t = v-s-t' \in D(v,(d-1)+(d-1)) = D(v,\ell)$.
    Let $z \in K^{\binom{2d}{2}}$ be the vector consisting of all $G_w$ for $w \in D(v,\ell)$,
    for some convenient ordering of $D(v,\ell)$.

    \item
    Let $\bar F_b \coloneqq F_{b e_j + (d-b) e_k}$ for $0 \leq b \leq d$;
    these are the coefficients $F_t$ appearing on the left hand side of \eqref{equation:system-v3}.
    Let $A$ be the $2d \times 2d$ matrix (over $R$) given by
    \begin{equation*}
      A = \begin{pmatrix}
            \bar F_0 & \bar F_1 & \bar F_2 & \cdots   & \cdots   & \bar F_d & \\
                     & \ddots   &          &          &          &          & \ddots \\
                     &          & \bar F_0 & \bar F_1 & \bar F_2 &  \cdots  & \cdots & \bar F_d \\
            0        & \bar F_1 & 2\bar F_2 & \cdots   & \cdots   & d\bar F_d & \\
                     & \ddots   &          &          &          &          & \ddots \\
                     &          & 0        & \bar F_1 & 2\bar F_2 &  \cdots  & \cdots & d\bar F_d \\
          \end{pmatrix}.
    \end{equation*}
    The columns correspond to the unknowns $y_c$ for $0 \leq c \leq 2d-1$.
    The first group of $d$ rows corresponds to the first equation in \eqref{equation:system-v3},
    and the second group to the second equation.
    The rows in each group are indexed by $a = 0, \ldots, d-1$,
    corresponding to the values of $s \in E$ via $s = a e_j + (d - 1 - a)e_k$.

    \item
    Let $M_{v,m}$ be the $2d \times \binom{2d}{2}$ matrix encoding the linear combinations on the right hand side of \eqref{equation:system-v3}.
    The columns of $M_{v,m}$ correspond to the known values $G_w$ for $w \in D(v,\ell)$,
    and the rows to the $2d$ equations.
    More explicitly, in the first $d$ rows, indexed by $a = 0, \ldots, d-1$,
    we place the value $(m+1)t_i - (v_i+1)$ in the column corresponding to $v + e_i - s - t$
    for each $t = t' + e_i$, $t' \in D_{d-1}$.
    Similarly, in the last $d$ rows, we place the values $(v_j - s_j)t_i - (v_i+1)t_j$ in suitable positions.
    The entries of $M_{v,m}$ may be regarded as linear polynomials in $R[v,m]$.
  \end{itemize}
  With these definitions, the system \eqref{equation:system-v3} may be expressed compactly as
  \begin{equation}
     \label{equation:matrix-system}
     (v_i + 1) A y = M_{v,m} z.
  \end{equation}

  The matrix $A$ is the Sylvester matrix of $F_{x_i=0, x_k=1}$ and  $(\partial_j F)_{x_i=0, x_k=1}$ as degree~$d$ polynomials in~$x_j$.
  By Proposition 1.8 in \cite[p.\,435]{GKZ94} we have
  \begin{equation*}
      \det A = \pm F_{d e_j} F_{d e_k}\! \disc_{x_j} F_{x_i=0, x_k=1} = \pm \Delta_d^* \left( F_{x_i = 0} \right) \neq 0.
  \end{equation*}
  We may therefore solve the system explicitly as follows.
  Define
  \begin{equation}
    \label{equation:theta-defn}
    \theta_i \coloneqq \det A,
  \end{equation}
  and let $\adj(A)\in R^{2d\times 2d}$ denote the matrix satisfying $\adj(A)A=(\det A)I$.
  Multiplying \eqref{equation:matrix-system} by $\adj(A)$ on the left yields the solution
  \begin{equation*}
    (v_i+1) \theta_i y = \adj(A) M_{v,m} z.
  \end{equation*}
  Note that the columns of $\adj(A) M_{v,m}$ correspond to monomials $u \in D(v,\ell)$,
  and the rows correspond to monomials $w \in S \subseteq D(v+e_i,\ell+1)$,
  i.e., the $c$-th row corresponds to $w = v+e_i-ce_j-(2d-1-c)e_k$ for $0 \leq c \leq 2d-1$.
 
  Finally we show how to define the matrix for the desired map $\phi_{v,i} \colon \calD_{v,\ell} \to \calD_{v+e_i,\ell+1}$.
  For $w \in D(v+e_i,\ell+1)$ and $u \in D(v,\ell)$, the matrix entry $(\phi_{v,i})_{w,u}$ is given by
  \begin{equation}\label{eq:phivi}
    (\phi_{v,i})_{w,u} = \begin{cases}
      (v_i + 1) \theta_i \delta_{w,u}, & \text{if } w \in D(v,\ell), \\
      (\adj(A) M_{m,v})_{w,u},         & \text{if } w \notin D(v,\ell),
      \end{cases}
  \end{equation}
  where $\delta_{w,u}$ if $w=u$ and $0$ otherwise.
\end{proof}

\begin{remark}
  One may attempt to apply the construction in the proof of Lemma \ref{lemma:construction phi}
  for values of $\ell$ other than $2d - 2$.
  This leads to a system of $2(\ell - d + 2)$ equations in $\ell + 2$ unknowns.
  Ultimately, the reason we work with $\ell = 2d-2$ is that this is the smallest
  value of $\ell$ for which there are at least as many equations as unknowns.
\end{remark}

\begin{remark}
  As observed in Lemma~\ref{lemma:calE inclusions} the equations defining $\calE_{v,\ell}$ are a subset of the equations defining $\calE_{v + e_i, \ell + 1}$.
  In the setup of Lemma~\ref{lemma:construction phi} this difference of equations has size~$2d$.

  The condition $\Delta_d(F_{x_i=0}) \neq  0$ ensures that these $2d$ equations are linearly independent.
  Furthermore, if $v_i + 1 \neq 0$, then given $G \in \calE_{v, \ell}$ there is a unique $\widetilde{G} \in \calE_{v + e_i, \ell + 1}$ such that $\restr{\widetilde{G}}{D(v,\ell)} = (v_i + 1) \theta_i G$.
  Thus when $v_i + 1 \neq 0$, we have $\phi_{v,i}(\calE_{v,\ell}) = \calE_{v + e_1, \ell + 1}$.
\end{remark}

For the remainder of this section we fix distinct $i,j\in \{0,1,2\}$.
By composing the map $\phi_{v,i} \colon \calD_{v, \ell} \to \calD_{v + e_i, \ell + 1}$ with the projection $P_{v + e_i,j} \colon \calD_{v + e_i, \ell + 1} \twoheadrightarrow \calD_{v + e_i - e_j, \ell}$ we obtain the map
\begin{equation}
\tau_{v,i,j} \coloneqq P_{v+e_i,j}\circ \phi_{v,i} \colon \calD_{v, \ell} \to \calD_{v + e_i - e_j, \ell},
\end{equation}
and the diagram~\eqref{equation:calD cd} as desired.
We now check that $\tau_{v,i,j}$ has the desired properties.  In particular, if $G \in \calE_{v, \ell}$, then $\tau_{v,i,j}(G) \in \calE_{v + e_i - e_j, \ell}$, meaning that $\tau_{v,i,j}(G)$ satisfies the equations on a shifted set of monomials.

\begin{corollary}\label{corollary:taumap}
  We have $\tau_{v,i,j}(\calE_{v, \ell}) \subseteq \calE_{v + e_i -e_j, \ell}$ and
  the composition
  \begin{equation*}
    \begin{tikzcd}[row sep=1pc,column sep=3.5pc]
    \calD_{v - e_j, \ell - 1}
    \arrow[r, hookrightarrow]
    &
    \calD_{v, \ell}
    \arrow[r, "\tau_{v,i,j}"]
    &
    \calD_{v + e_i - e_j ,\ell}
    \arrow[r, twoheadrightarrow, "P_{v + e_i,j}"]
    &
    \calD_{v - e_j, \ell - 1}
  \end{tikzcd}
  \end{equation*}
  is scalar multiplication by $(v_i + 1) \theta_i$, and $\tau_{v,i,j}$ is invertible when $v_i+1\ne 0$ in $R$.

  The map $\tau_{v, i, j}$ may be represented by a $\binom{2d}{2} \times \binom{2d}{2}$ matrix whose entries are $R$-linear combinations of~$1, v_0, v_1, v_2$ corresponding to linear polynomials in~$R[v]$.
\end{corollary}
\begin{proof}
  The first part follows by the definition of $\tau_{v,i,j}$ combined with Lemmas~\ref{lemma:construction phi} and \ref{lemma:calE inclusions}.
  The last part also follows from Lemma~\ref{lemma:construction phi}, where we note that $\#D(v,\ell)=\#D(v+e_i-e_j,\ell)=\#D_\ell=\binom{\ell+n}{n} = \binom{2d}{2}$ for $n=2$ and $\ell=2d-2$.
\end{proof}

Let $\phi^\calE_{v,i} \colon \calE_{v, \ell} \to \calE_{v + e_i, \ell + 1}$ be the restriction of $\phi_{v,i}\colon \calD_{v, \ell} \to \calD_{v + e_i, \ell + 1}$ and similarly define $\tau_{v,i,j}^\calE$ and $P_{v,i}^\calE$.
Because we have assumed that $F$ is nondegenerate and $m+1\ne 0$ in $R$, applying Corollary~\ref{corollary:dim = d^2} with $\ell=2d-2$ and $\ell+1=2d-1$ yields
\begin{equation}\label{eq:dimd2}
\dim_K \calW_\ell = \dim_K \calE_{v,\ell}=\dim_K \calE_{v+e_i,\ell+1} = \dim_K \calE_{v+e_i-e_j,\ell} = d^2.
\end{equation}

Since $\dim_K \calE_{v,\ell}=\dim_K \calW_\ell$, by \eqref{eq:dimd2}, Corollary~\ref{corollary:upper bound rank Ev} gives us bijections
\begin{equation}
    \pi_{v,\ell}^\calE \colon \calE_{v, \ell} \to \calW_{v, \ell}, \qquad
    \psi_{v,\ell}^\calE \colon \calW_{v, \ell} \to \calE_{v, \ell},
\end{equation}
which are the restrictions of $\pi_{v, \ell}$ and $\psi_{v, \ell}$, respectively.
We now consider the map
\begin{equation}\label{eq:Tvij}
  T_{v,i,j} \coloneqq \pi_{v + e_i - e_j, \ell}^\calE \circ \tau_{v, i, j}^\calE \circ \psi_{v, \ell}^\calE
  \colon \calW_{v, \ell} \longrightarrow \calW_{v + e_i - e_j, \ell}.
\end{equation}
In other words, the map  $T_{v,i,j}$ re-expresses the shifting map $\tau_{v,i,j}^\calE$ in terms of a basis for $\calW_{v,\ell}$.
We are interested in applying chains of such maps $T_{v + \bullet,i,j}$, thus for any $s>0$ we define
\begin{equation}\label{eq:Tvijs}
  T_{v,i,j}^s \coloneqq \prod_{s > k \ge 0} T_{v + k (e_i - e_j),i,j} = T_{v + (s-1) e_i - (s-1) e_j,i,j} \circ \cdots \circ T_{v + e_i - e_j,i,j} \circ  T_{v,i,j},
\end{equation}
where the product is taken over decreasing values of $k$ starting from $s-1$; note that the symbol $s$ in $T_{v,i,j}^s$ is a superscript, not an exponent.

\begin{corollary}\label{corollary:calT^s}
    Let $s$ be a positive integer. We have
    \begin{equation*}
       T^s_{v,i,j} = (m + 1)^{s-1} \lambda_\ell^{s-1}
       \pi_{v + s e_i -  s e_j, \ell}^\calE \circ
      \left( \prod_{s > k \ge 0} \tau_{v + k( e_i -  e_j), i, j}^\calE \right) \circ \psi_{v, \ell}^\calE.
    \end{equation*}
    Furthermore, $\pi_{v + s e_i -  s e_j, \ell}^\calE \circ \left(\prod_{s > k \ge 0} \tau_{v + k( e_i -  e_j), i, j}^\calE \right) \circ \psi_{v, \ell}^\calE$
    may be represented by $d^2\times d^2$ matrix whose entries are polynomials in $R[v,m] = R[v_0, v_1, v_2, m]$ of degree at most $s+1$.

\end{corollary}
\begin{proof}
  Lemma~\ref{lemma:compression} implies $\psi_{v + k (e_i - e_j),\ell}^\calE \circ \pi_{v + k (e_i - e_j), \ell}^\calE = (m + 1) \lambda_\ell\id_{\calE_{v+k(e_i-e_j),\ell}}$ for $0\le k < s$.  Applying this repeatedly yields
  \begin{equation}\label{equation:calT^s}
    \begin{aligned}
      T^s_{v,i,j}
      &\coloneqq \prod_{s > k \ge 0} T_{v + k (e_i -  e_j),i,j}\\
      &= \prod_{s > k\ge 0} \pi_{v + (k + 1) (e_i -  e_j), \ell}^\calE \circ \tau_{v +  k(e_i -  e_j), i, j}^\calE \circ \psi_{v +  k(e_i -  e_j), \ell}^\calE
      \\
      &= \pi_{v + s (e_i - e_j), \ell}^\calE \circ
      \left( \prod_{s > k > 0} \tau_{v + k (e_i -  e_j), i, j}^\calE \circ \psi_{v + k (e_i -  e_j), \ell}^\calE\circ \pi_{v + k( e_i -  e_j), \ell}^\calE \right)
      \circ \tau_{v, i, j}^\calE \circ \psi_{v, \ell}^\calE
      \\
      &= (m + 1)^{s-1} \, \lambda_\ell^{s-1} \pi_{v + s e_i -  s e_j, \ell}^\calE \circ
      \left(\prod_{s > k \ge 0} \tau_{v + k( e_i -  e_j), i, j}^\calE \right)
      \circ \psi_{v, \ell}^\calE.
    \end{aligned}
  \end{equation}
 Lemma~\ref{lemma:compression}, Corollary~\ref{corollary:upper bound rank Ev}, and Corollary~\ref{corollary:taumap} imply that the RHS can be represented as the product of a scalar, a $d^2\times \binom{2d}{2}$ matrix, $s-1$ different $\binom{2d}{2}\times \binom{2d}{2}$ matrices, and a $\binom{2d}{2}\times d^2$ matrix, all of whose entries are linear polynomials in $R[v,m]$.
The corollary follows.
\end{proof}

Corollary~\ref{corollary:calT^s} combined with Lemma~\ref{lemma:compression} yields the following corollary.

\begin{corollary}\label{corollary:calT to sol}
   Let $s\in \Z_{\ge 0}$ and let $G\in R[x]_{dm}$ satisfy equation~\eqref{equation:GF}.
   Then,
  \vspace{-8pt}
  \[
    T^s_{v,i,j} \circ \pi_{v, \ell}^\calE \left( \restr{G}{D(v,\ell)} \right)
    =
    \theta_i^{s} \lambda_\ell^s  (m + 1)^s \left(\prod_{k=1}^{s} (v_i + k)\right)  \pi_{v + s(e_i - e_j), \ell}^\calE \left( \restr{G}{D(v + s(e_i - e_j),\ell)} \right).
  \]
\end{corollary}

\noindent
Before stating the final result of this section, we remind the reader of our running assumptions:
\begin{itemize}
\item $i,j \in \{0,1,2\}$ distinct;
\item $R=\Z$ or $\F_p$, $n=2$, $d>1$, $\ell=2d-2$, $m>0$, and $d(m+1)\ne 0$ in $R$;
\item $F\in R[x]_d$ is nondegenerate, meaning $\Delta_d^*(F)\ne 0$ (see \eqref{eq:discstar} for the definition of $\Delta_d^*$).
\end{itemize}

\begin{theorem}\label{theorem:calT mod p}
  Let $p$ be a prime that is equal to the characteristic of $R$ when $R=\F_p$ and does not divide $\Delta_d^*(F)d(m+1)$ when $R=\Z$. Let~$s$ be a positive integer, and let $G\in R[x]_{dm}$ satisfy equation~\eqref{equation:GF}.
  The following hold:
  \begin{enumerate}
    \item[\textup{(a)}]
       If $w \in \Z^{n+1}$ of degree $dm + \ell$ and $v \equiv w \bmod p$ then the matrices representing  $T^{s}_{v,i,j}$ and $T^{s}_{w,i,j}$ agree modulo $p$.
    \item[\textup{(b)}]
      If $v_i \equiv 0 \bmod p$ and $s = p-1$, then $(m + 1)^s \lambda_\ell^s \theta_i^s  \prod_{k=1}^{s} (v_i + k) \equiv -1 \bmod p$ and
      \[
        T^{p-1}_{v,i,j} \circ \pi_{v, \ell}^\calE \left( \restr{G}{D(v,\ell)} \right) \equiv  -  \pi_{v + (p-1)(e_i - e_j), \ell}^\calE \left( \restr{G}{D(v + (p-1)(e_i - e_j),\ell)} \right) \bmod p.
      \]
      When $v_j \equiv -1 \bmod p$ also holds, the matrix $T^{p-1}_{v,i,j}$ is invertible modulo $p$ and its inverse is $T^{p-1}_{v + (p-1)(e_i - e_j),j,i}$.
  \end{enumerate}
\end{theorem}
\begin{proof}
  For (a) note that $T^{s}_{v,i,j}$ is representable as a matrix with entries in $R[v]$.
  For (b) we apply Fermat's little theorem and Wilson's theorem to obtain $\prod_{k=1}^{p-1} (v_i + k) \equiv (p-1)! \equiv -1 \bmod p$, which together with Corollary \ref{corollary:calT to sol} implies the first claim.
  For the second claim in (b), we apply $T^{p-1}_{v + (p-1)(e_i - e_j),j,i}$ to both sides of the first claim to obtain
  \begin{equation*}
    \begin{aligned}
       T^{p-1}_{v + (p-1)(e_i - e_j),j,i} \circ\  T^{p-1}_{v,i,j}\ \circ\ &\pi_{v, \ell}^\calE\! \left( \restr{G}{D(v,\ell)} \right)\\
       &\equiv - T^{p-1}_{v + (p-1)(e_i - e_j),j,i}\! \circ \pi_{v + (p-1)(e_i - e_j),\ell}^\calE\! \left( \restr{G}{D(v + (p-1)(e_i - e_j),\ell)} \right) \bmod p \\
       &\equiv \pi_{v,\ell}^\calE\left(\restr{G}{D(v,\ell)}\right) \bmod p,
    \end{aligned}
  \end{equation*}
  where the last equivalence follows from the first claim in (b), since $v_j\equiv -1\bmod p$ implies $(v + (p-1)(e_i - e_j))_j \equiv 0 \bmod p$, allowing us to apply the first claim to $T^{p-1}_{v + (p-1)(e_i - e_j),j,i}$.
\end{proof}

\section{Computing Cartier--Manin matrices of a smooth plane quartic}\label{sec:algorithms}

Let $X\colon f(x_0,x_1,x_2)=0$ be a smooth plane quartic defined by a nondegenerate homogeneous quartic $f\in R[x_0,x_1,x_2]_4$.  In this section we give algorithms to compute the Cartier--Manin matrix $A_p$ of $X$ when $R=\F_p$, or the Cartier--Manin matrices $A_p$ of the reductions of $X$ modulo primes $p\le N$ of good reduction up to a given bound $N$ when $R=\Z$.

We first consider the case $R=\F_p$, where $p$ is an odd prime, noting that for $p=2$ the Cartier--Manin matrix can be extracted directly from the coefficients of $f=f^{p-1}$ via \eqref{equation:cartier--main spq}.
We will apply the infrastructure developed in \S\ref{sec:shifting} with $F = f$ and $m = p -2$.
In particular, we work with $\ell = 6 = 2 d - 2$ and $dm + \ell = 4 (p-2) + 6 = 4 p -2$ throughout.

\begin{figure}
\begin{center}
\begin{tikzpicture}[scale=0.4]
\input{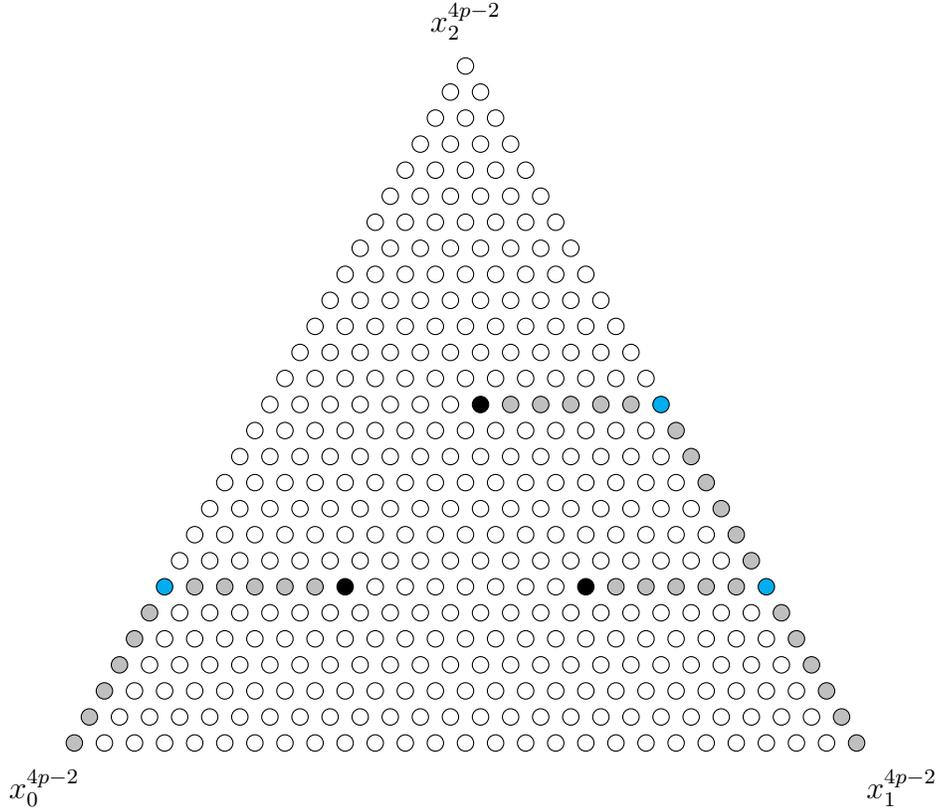}
\end{tikzpicture}
\end{center}
\caption{Illustration for $p=7$. The target points $v$ in the interior are shown in black with $v^{(1)}$ at the top center, $v^{(2)}$ at the lower left, and $v^{(3)}$ at the lower right. The intermediate points $w$ are in blue, and the paths used to reach the target points $v$ are shown in gray.}\label{fig:triangle}
\end{figure}

Let us first sketch our algorithm by working backwards from our goal.
The coefficients of $f^{p-1}$ that appear in the $i$th column of the matrix $A_p$ in \eqref{equation:cartier--main spq} lie in
$\restr{f^{p-1}}{D(v^{(i)}, 2)}$ for
\begin{equation}
  v^{(1)} \coloneqq (p-1, p, 2p-1),\quad v^{(2)}\coloneqq (2p, p-1, p-1),\quad v^{(3)}\coloneqq  (p-1, 2p, p-1);
 \end{equation}
note that the $v^{(i)}$ are not symmetric because the indices in the columns of \eqref{equation:cartier--main spq} are not.
Now $D_{v, 2} = B_{v,2}$, since $2<4= d$, so $\pi_{v,6}$ has codomain $\calW_{v,6}$ and it suffices to compute
\begin{equation}\label{equation:we only need pi}
  \pi_{v,6}\bigl(\restr{f^{p-2}}{D(v, 6)}\bigr) = \restr{f^{p-2}}{B(v, 6)} + \restr{f^{p-1}}{B(v, 2)} \in \calW_{v,6}
\end{equation}
for $v=v^{(1)},v^{(2)},v^{(3)}$.  We now define
\begin{equation}
 w^{(1)}\coloneqq (0,2p-1,2p-1),\quad w^{(2)}\coloneqq (3p-1, 0, p-1),\quad w^{(3)}\coloneqq (0,3p-1,p-1),
\end{equation}
with $w^{(1)}=v^{(1)}+(p-1)(e_1-e_0)$, $w^{(2)}=v^{(2)}+(p-1)(e_0-e_1)$, and $w^{(3)}=v^{(3)}+(p-1)(e_1-e_0)$.
Let $C_p \in \F_p^{16\times 16}$ denote the matrix representing the linear operator
\begin{equation}
    T_{w^{(1)},0,1}^{p-1} \colon \calW_{w^{(1)},6} \to \calW_{v^{(1)}, 6},
\end{equation}
determined by the nondegenerate polynomial $f\in \F_p[x_0,x_1,x_2]_4$.
By Theorem~\ref{theorem:calT mod p} (a), the matrix $C_p$ also represents
\begin{equation}
  T_{w^{(3)},0,1}^{p-1} \colon \calW_{w^{(3)},6} \to \calW_{v^{(3)}, 6},
\end{equation}
since $v^{(1)}\equiv v^{(3)}\bmod p$ and $w^{(1)}\equiv w^{(3)}\bmod p$, and by Theorem~\ref{theorem:calT mod p} (b), $C_p^{-1}$ represents
\begin{equation}
    \left(T_{v^{(2)},0,1}^{p-1}\right)^{-1} \equiv T_{w^{(2)},1,0}^{p-1} \colon \calW_{w^{(2)},6} \to \calW_{v^{(2)},6}
\end{equation}
since $v^{(2)}_0\equiv 0\bmod p$ and $v^{(2)}_1\equiv -1\bmod p$ and $w^{(2)}=v^{(2)}+(p-1)(e_0-e_1)$.
We can thus use the matrix $C_p$ and its inverse to traverse the three paths from the intermediate points $w$ depicted as blue dots on the exterior of triangle in Figure~\ref{fig:triangle} to the target interior points $v$.

To obtain the coefficients of $\restr{f^{p-2}}{D(w, 6)}$ for $w=w^{(1)},w^{(2)},w^{(3)}$ we could apply a variation of the method of \S\ref{sec:shifting} for $n = 1$ (each $w$ has a zero entry we can ignore), but we prefer to use a simpler approach that we illustrate for $w=w^{(3)}$.
Let $h(t)\coloneqq f(0,1,t)$.  Then
\begin{equation}\label{equation:F^{p-eps} mod p}
  h^{p-2}(t) \equiv h(t^p) h^{-2}(t).
\end{equation}
If we put $g(t)\coloneqq h(t)^2=\sum_{i=0}^8a_it^i$ and let
\begin{equation}
    a_0/g(t) = \sum_{i\geq 0} c_i t^i \in \F_p[[t]],
\end{equation}
then we can compute $(c_{s}, c_{s-1}, \cdots, c_{s-7})$ as the first column of
$Q_g^s$, where
\begin{equation}
  a_0 Q_g \coloneqq
  \begin{bmatrix}
    -a_1 & -a_2 &  -a_3 & \cdots & -a_8\\
    a_0  & 0 & 0 & \cdots & 0\\
    0 & a_0  & 0 & \cdots & 0\\
    \vdots &  & \ddots  & & \vdots \\
    0 & \cdots  &  0 & a_0 & 0
  \end{bmatrix}
  ,
\end{equation}
Computing $Q_g^s$ with $s=p-1$ allows us to derive the $\binom{6+1}{1}=7$ coefficients of $\restr{f^{p-2}}{D(w, 6)}$ we need using $c_s,\ldots,c_{s-6}$; the other $\binom{6+2}{2}-\binom{6+1}{1}=21$ coefficients correspond to monomials in $\F_p[x^{\pm}]$ that contain  a negative exponent and are necessarily zero because $f^{p-2}$ is a polynomial.  In terms of Figure~\ref{fig:triangle}, the computation we have just described corresponds to walking $p-1$ steps along the gray path from the lower right corner of the triangle to the first blue dot on the right edge (the 21 zero coefficients correspond to monomials outside the triangle).

The cases $w=w^{(1)},w^{(2)}$ are treated similarly using suitable $g(t)$ and $s$.

\begin{algorithm}\label{algorithm:cartier manin}

\noindent
Given a nondegenerate $f \in \F_p[x_0,x_1,x_2]_4$ and the corresponding matrix $C_p \in \F_p^{16\times 16}$, compute the Cartier--Manin matrix of $X\colon f(x_0,x_1,x_2)=0$ as follows:

\begin{enumerate}
\setlength{\itemsep}{8pt}
\item Compute $\restr{f^{p-2}}{D(w,6)}$ for $w=w^{(1)},w^{(2)},w^{(3)}$ (the blue dots in Figure~\ref{fig:triangle}) using suitably chosen $g\in \F_p[t]$ and $Q_g^s\in \F_p^{8\times 8}$ as described above:
\smallskip

\begin{enumerate}
\item Compute the edge coefficients of $\restr{f^{p-2}}{D(w^{(1)}, 6)}$ using $g(t) \coloneqq f(0,1,t)^2$:
\[
\bigl(f^{p-2}_{w^{(1)}-je_2-(6-j)e_1}\bigr)_{0\le j \le 7} = \left(f_{(0,3,1)} {f_{(0,4,0)}}^{-2} Q_g ^{p-1}  + {f_{(0,4,0)}}^{-1} Q_g ^{2 p-1} \right)\cdot (1, 0, \ldots, 0)^T.
\]

\item Compute the edge coefficients of $\restr{f^{p-2}}{D(w^{(2)}, 6)}$ using $g(t) \coloneqq f(1,0,t)^2$:
\[
\bigl(f^{p-2}_{w^{(2)}-je_2-(6-j)e_0)}\bigr)_{0\le j\le 7} = {f_{(4,0,0)}}^{-1} Q_g ^{p-1} \cdot (1, 0, \ldots, 0)^T.
\]

\item Compute the edge coefficients of $\restr{f^{p-2}}{D(w^{(3)}, 6)}$ using $g(t) \coloneqq f(0,1,t)^2$:
\[
\bigl(f^{p-2}_{w^{(3)}-je_2-(6-j)e_1}\bigr)_{0\le j \le 7} = {f_{(0,4,0)}}^{-1} Q_g ^{p-1} \cdot (1, 0, \ldots, 0)^T.
\]
\end{enumerate}

\item Compute $\pi_{v,6}(\restr{f^{p-2}}{D(v, 6)})$ for $v=v^{(1)},v^{(2)},v^{(3)}$ (the black dots in Figure~\ref{fig:triangle}) using Theorem~\ref{theorem:calT mod p} and Equation~\eqref{equation:we only need pi} as follows:
\smallskip

\begin{enumerate}
\item Compute the first column of $A_p$ using $v^{(1)} = (p-1, p, 2p-1)$:
\begin{align*}
\left(f^{p-1}_{(p-1,p,2p-3)},\boldsymbol{f^{p-1}_{(p-1,p-1,2p-2)}},\boldsymbol{f^{p-1}_{(p-1,p-2,2p-1)}},f^{p-1}_{(p-2,p,2p-2)},\boldsymbol{f^{p-1}_{(p-2,p-1,2p-1)}},f^{p-1}_{(p-3,p,2p-2)}\right)\\
= \restr{\Bigl(\pi_{v^{(1)},6} \bigl(\restr{f^{p-2}}{D(v^{(1)}, 6)} \bigr)\Bigr)}{B(v^{(1)},2)} =  - C_p \circ \pi_{w^{(1)},6} \bigl( \restr{f^{p-2}}{D(w^{(1)}, 6)}\bigr).
\end{align*}

\item Compute the second column of $A_p$ using $v^{(2)} = (2p, p-1, p-1)$:
\begin{align*}
\left(f^{p-1}_{(2p,p-1,p-3)},f^{p-1}_{(2p,p-2,p-2)},f^{p-1}_{(2p,p-3,p-1)},\boldsymbol{f^{p-1}_{(2p-1,p-1,p-2)}},\boldsymbol{f^{p-1}_{(2p-1,p-2,p-1)}},\boldsymbol{f^{p-1}_{(2p-2,p-1,p-1)}}\right)\\
=\restr{\Bigl(\pi_{v^{(2)},6} \bigl(\restr{f^{p-2}}{D(v^{(2)}, 6)} \bigr)\Bigr)}{B(v^{(2)},2)}  =  - {C_p}^{-1} \circ \pi_{w^{(2)},6} \bigl( \restr{f^{p-2}}{D(w^{(2)}, 6)}\bigr).
\end{align*}

\item Compute the third column of $A_p$ using $v^{(3)} = (p-1, 2p, p-1)$:
\begin{align*}
\left(f^{p-1}_{(p-1,2p,p-3)},\boldsymbol{f^{p-1}_{(p-1,2p-1,p-2)}},\boldsymbol{f^{p-1}_{(p-1,2p-2,p-1)}},f^{p-1}_{(p-2,2p,p-2)},\boldsymbol{f^{p-1}_{(p-2,2p-1,p-1)}},f^{p-1}_{(p-3,2p,p-1)}\right)\\
=\restr{\Bigl(\pi_{v^{(3)},6}\bigl(\restr{f^{p-2}}{D(v^{(3)}, 6)}\bigr)\Bigr)}{B(v^{(3)},2)} =  - C_p \circ \pi_{w^{(3)},6} \bigl( \restr{f^{p-2}}{D(w^{(3)}, 6)} \bigr).
\end{align*}
\end{enumerate}

\item Output the matrix $A_p\in \F_p^{3\times 3}$ defined in \eqref{equation:cartier--main spq} using the coefficients of $f^{p-1}$ that are shown in bold above.
\end{enumerate}

\end{algorithm}

\begin{remark}
The matrix $Q_g^{p-1}$ in step (1c) is the same as in step (1a) and need not be recomputed.  The matrices that represent $\pi_{w,6}$ for $w=w^{(1)},w^{(2)},w^{(3)}$ in step (2) are all the same, since $\pi_{w,6}$ does not depend on $w$, by Lemma~\ref{lemma:compression}.  Indeed, if $\iota(t)\in \{1,\ldots \#D_\ell\}$ is the index of $t\in D_\ell$ in its lexicographic ordering, the matrix $W\in R^{16\times 28}$ with nonzero entries $W_{\iota(u),\iota(t+u)}\coloneqq F_t$ for $u\in D_2$ and $t\in D_4$ and $W_{6+j,18+j}\coloneqq 1$ for $1\le j\le 10$ represents $\pi_{w,6}$.
\end{remark}

\begin{remark}\label{rem:uncompressedCM}
If we instead use the ``uncompressed'' matrix $U_p\in \F_p^{28\times 28}$ representing the linear operator 
$\prod_{p - 1> k \ge 0} \tau_{w^{(1)} + k( e_0 -  e_1), 0, 1}^\calE$, which by \eqref{corollary:calT^s} satisfies
\[
U_p = -\lambda_6^{-1}\psi_{v^{(1)},6}^\calE \circ C_p \circ \pi_{w^{(1)},6}^\calE,
\]
we can consider an ``uncompressed'' version of Algorithm~\ref{algorithm:cartier manin}.  We replace $C_p\circ \pi_{w^{(1)},6}$ and $C_p\circ \pi_{w^{(3)},6}$ with $\pi_{v^{(1)}} \circ U_p$ and $\pi_{v^{(3)}} \circ U_p$, respectively, to obtain the desired vectors in (2a) and (2c), and for (2b) we replace $C_p^{-1}$ with $-\lambda_6 (\pi_{v^{(1)},6}^\calE \circ U_p\circ \psi_{w^{(1)},6}^\calE)^{-1}$.
\end{remark}

\begin{lemma}
Algorithm~\ref{algorithm:cartier manin} runs in $O(\log^2 p\log\log p)$ time using $O(\log p)$ space.
\end{lemma}
\begin{proof}
The algorithm uses $O(\log p)$ ring operations for the matrix exponentiations and $O(1)$ field inversions in step (1), and $O(1)$ field operations in step (2).  Each ring operation in~$\F_p$ can be achieved using $O(1)$ ring operations in $\Z$ on integers with $O(\log p)$ bits (using Newton iteration to perform fast Euclidean division, see \cite[Thm.\,9.8]{GG13}), which yields a bit complexity of $O(\M(\log p)) = O(\log p \log\log p)$ per ring operation via \cite{HvdH21}.
We can perform field inversions in $O(\M(\log p)\log\log p)=O(\log p (\log \log p)^2)$ time using a fast GCD algorithm \cite[Cor.\,11.13]{GG13}, which is dominated by the cost of $O(\log p)$ ring operations; the time bound follows and the space bound is immediate.
\end{proof}

We now give our algorithms to compute the Cartier--Manin matrix of a smooth plane quartic.  Let us define the matrix
\begin{equation}\label{eq:Mt}
M(t) \coloneqq T_{w^{(1)}+t(e_0-e_1),0,1}\in R[t]^{16\times 16},
\end{equation}
whose entries are polynomials in $t$ of degree at most 2, by Corollary~\ref{corollary:calT^s}.  From \eqref{eq:Tvij}, we see that $M(t)$ can be computed as the product of matrices in $R[t]^{16\times 28}, R[t]^{28\times 28}, R[t]^{28\times 16}$ representing the maps $\pi^{\calE}_{v(t)+e_0-e_1,6}$, $\tau^\calE_{v(t),0,1}$, $\psi^\calE_{v(t),6}$, respectively, where $v(t)=w^{(1)}+t(e_0-e_1)$.
The matrices representing $\pi^{\calE}_{v(t)+e_0-e_1,6}$  and $\psi^\calE_{v(t),6}$ are computed as in the proof of Lemma~\ref{lemma:compression}: the matrix representing  $\pi^{\calE}_{v(t)+e_0-e_1,6}$ is independent of $v(t)$, its entries are coefficients of~$f$ or elements of $\{0,1\}$, while the entries in $\psi^\calE_{v(t),6}$ are determined by \eqref{equation:phi definition}.
The matrix representing $\tau^\calE_{v(t),0,1}=P_{v(t)+e_0,j}^\calE\circ \phi_{v(t),i}^\calE$ is computed by composing the matrix in $\{0,1\}^{28\times 36}$ representing the projection $P_{v(t)+e_0,1}^\calE$ with the matrix in $R[t]^{36\times 28}$ representing $\phi_{v,i}^{\calE}$ whose entries are given by \eqref{eq:phivi}.  From equation \eqref{eq:Tvijs} we then have
\begin{equation}\label{eq:Cp}
C_p \coloneqq T^{p-1}_{w^{(1)},0,1} =\!\!\! \prod_{p-1>j\geq 0} \!\!\! M(j)\bmod p\in \F_p^{16\times 16}.
\end{equation}

\begin{algorithm}\label{algorithm:mod p}
Given a nondegenerate  $f \in \F_p[x_0,x_1,x_2]_4$, compute the Cartier--Manin matrix $A_p$ of the smooth plane quartic $X\colon f(x_0,x_1,x_2)=0$ as follows:

\begin{enumerate}
\setlength{\itemsep}{3pt}
\item Compute the matrix $M(t)\in \F_p[t]^{16\times 16}$ corresponding to $f$ as described above.
\item Compute the matrix $C_p= T_{w_1,0,1}^{p-1} = \prod_{p-1>j\geq 0} M(j)\in \F_p^{16\times 16}$.
\item Use Algorithm~\ref{algorithm:cartier manin} with inputs $C_p$ and $f$ to compute the Cartier--Manin matrix $A_p$.
\end{enumerate}
\end{algorithm}

\begin{remark}\label{rem:uncompressedmodp}
We may also consider an uncompressed version of Algorithm~\ref{algorithm:cartier manin} that uses $M(t)\coloneqq \tau_{w^{1}+t(e_0-e_1),0,1}^\calE\in R[t]^{28\times 28}$ to compute the matrix $U_p$ defined in Remark~\ref{algorithm:cartier manin} rather than using the matrices $M(t)$ defined in \eqref{eq:Mt} to compute $C_p$. Note that in the former case the entries of $M(t)$ have degree at most 1 rather than 2.
\end{remark}

\begin{theorem}\label{thm:modp}
Algorithm~\ref{algorithm:mod p} can be implemented to use $O(p\log p\log\log p)$ time and $O(\log p)$ space, and also to use $O(p^{1/2}\log^2\! p)$ time and $O(p^{1/2}\log p)$ space.
\end{theorem}
\begin{proof}
The first complexity bound is achieved by iteratively instantiating the entries of $M(t)$ at $t=k$ and accumulating the matrix product in $C_p$.  This involves $O(p)$ ring operations in~$\F_p$, which takes $O(p\log p\log\log p)$ time using $O(\log p)$ space. The second complexity bounds is achieved by using the interpolation/evaluation algorithm of Bostan--Gaudry--Schost \cite{BGS07} to compute $\prod_{p-1>j\geq 0} M(j)$, which uses $\M(p^{1/2}\log p) = O(p^{1/2}\log^2 p)$ time and $O(p^{1/2}\log p)$ space.
The cost of invoking Algorithm~\ref{algorithm:cartier manin} in step (2) is negligible in both cases.
\end{proof}

\begin{remark}\label{rem:lineartime}
In our $O(p\log p\log\log p)$ implementation, rather than computing $C_p$ as the product of $p-1$ matrices $M(j)$, we instead iteratively multiply the vectors $\pi_{w^{(i)},6}( \restr{f^{p-2}}{D(w^{(i)}, 6)}\bigr)$ that appear in steps (2a) and (2c) of Algorithm~\ref{algorithm:cartier manin} by each matrices $M(j)$ as it is computed.  We then repeat this process using the curve defined by $f(x_1,x_0,x_2)$ to obtain the vector computed in step (2b); note that in steps (1c) and (2c) of Algorithm~\ref{algorithm:cartier manin} are identical to steps (1b) and (2b) except the roles of $x_0$ and $x_1$ are reversed.  This effectively replaces each matrix multiplication with 3 matrix-vector multiplications and is practically faster in the range of $p$ we consider.  The matrices $M(j)$ for $j=0,\ldots,p-1$ can be efficiently enumerated using finite differences (the entries of $M(t)$ are polynomials of degree at most 2).
\end{remark}

We now turn to the case $R=\Z$, where our strategy is to use an average polynomial-time approach to simultaneously compute the matrices $C_p$ at suitable primes $p\le N$ using a single matrix $M(t)\in \Z[t]^{16\times 16}$.  A nondegenerate polynomial $f\in\Z[x_0,x_1,x_2]_4$ will have nondegenerate reduction modulo all primes $p$ that do not divide $\Delta_4^*(f)$, but in order to obtain a valid matrix $C_p$ to use as input for Algorithm~\ref{algorithm:cartier manin} computed via \eqref{eq:Cp} with $M(t)\in \Z[t]^{16\times 16}$ we also need to ensure that the scalar $(m+1)\lambda_6$ arising in Lemma~\ref{lemma:compression} and the degree $d = 4$ are both nonzero modulo $p$.

Now $m+1=p-1$ is never divisible by $p$, so it suffices to restrict our attention to odd primes that do not divide $\lambda_6$. We thus define $D\coloneqq 2\lambda_6 \Delta_4 ^*(f)$ and treat all primes $p\le N$ that do not divide $D$ using an average polynomial-time approach and handle good primes $p\mid D$ as special cases via Remark~\ref{rem:badp} below.  The primes $p \mid D$ are bounded by a constant that does not depend on $N$, thus the time spent handling the good $p\mid D$ has no impact on the complexity of our algorithm as a function of $N$ (and it is completely negligible in practice).

\begin{remark}\label{rem:badp}
For primes $p\mid D$ where $f$ has good reduction we can compute the Cartier--Manin matrix directly from its definition, but we can more efficiently treat $p \nmid \Delta_4^*(f)$ by simply applying Algorithm~\ref{algorithm:mod p} to the nondegenerate reduction of $f$ modulo $p$.  In our implementation we do the same for good primes $p\mid \Delta_4^*(f)$ greater than 3 by applying a random linear transformation to the reduction of $f$ modulo $p$ until we obtain a nondegenerate polynomial $\widetilde f \in \F_p[x_0,x_1,x_2]$ that defines an isomorphic curve. For $p>3$ such a nondegenerate polynomial is guaranteed to exist by Proposition 3.2 of \cite{CV10}, and in practice we can find one quickly.  Note that we have assumed $f(x_0,x_1,x_2)=0$ is a model for $X$ that is smooth a $p$, but if not, replace $f$ modulo $p$ with the reduction of a model for $X$ that is smooth at $p$.
\end{remark}

Before describing our average polynomial-time algorithms to compute $A_p$ for $p\le N$ coprime to~$D$, we briefly recall some background material on remainder trees and forests.
Given a sequence of integer matrices $M_0,\ldots, M_{N-1}$ and a sequence of coprime integers $m_0,\ldots,m_{N-1}$ we wish to compute the following sequence of reduced partial products for $0\le k < N$:
\[
P_k\coloneqq M_0\cdots M_k\bmod m_k.
\]
Let $M_{-1}\coloneqq M_N\coloneqq m_N\coloneqq 1$, and for $0\le k < N/2$ let $M_k'\coloneqq M_{2k-1}M_{2k}$ and $m_k'\coloneqq m_{2k}m_{2k+1}$.
If we now recursively compute $P_k'\coloneqq M_0'\cdots M_k'\bmod m_k' = M_0\cdots M_{2k}\bmod m_{2k}m_{2k+1} $ for $0\le k < N/2$, we can then compute
\[
P_{2k} = P_{k}'\bmod m_{2k}\qquad\text{and}\qquad P_{2k+1} =  P_k'M_{2k+1} \bmod m_{2k+1}.
\]
Unwinding this recursion yields the \textsc{RemainderTree} algorithm described in \cite{HS14}.

The \textsc{RemainderForest} algorithm in \cite{HS16} reduces the time and (especially) the space needed by splitting the remainder tree into $2^\kappa$-subtrees, for a suitable choice of $\kappa$.  In \cite{HS14,HS16,Sut20} the \textsc{RemainderForest} algorithm is used to compute the sequence of vectors $V_k\coloneqq V_0M_0\cdots M_k\bmod m_k$ using vector-matrix multiplications to carry results from one subtree to the next, but it can also be used to compute $P_k= I M_0\cdots M_k\bmod m_k$ using the same approach. Below we record a special case of \cite[Theorem 3.3]{HS16}, in which $\|M_k\|$ denotes the logarithm of the largest absolute value appearing in the nonzero matrix $M_k$.

\begin{theorem}\label{thm:forest}
Fix a constant $c>0$.
Let $N$ be a positive integer, let $m_0,\ldots,m_{N-1}$ be positive coprime integers with $\log \prod_{k=0}^n m_k\le c n$ for $2\le n< N$, let $M_0,\ldots,M_{N-1}\in \Z^{r\times r}$ be nonzero integer matrices with $ r < c\log N$ and $\|M_i\|\le c\log N$.
We can compute the matrices
\[
P_k\coloneqq \prod_{i=0}^k M_i\bmod m_k
\]
for $0\le k < N$ in $O(r^2 N\log^3\!N)$ time using $O(r^2N)$ space.
\end{theorem}
\begin{proof}
We apply \cite[Thm.~3.3]{HS16} with $\kappa \coloneqq \lfloor 2\log_2\log_2 N\rfloor $, $B=cN$, $B'=1$, $H=c\log N$.  We use $\M(n)=O(n\log n)$ from \cite{HvdH21} and note that replacing $\M(n)$ with $n\log n$ in the statement of \cite[Lem.\,4]{HvdH18} allows us to omit the last step of the proof where the hypothesis that $\M(n)/(n\log n)$ is increasing is used and remove that hypothesis.

Provided $\log r = O(\log B)$, the complexity of multiplying $r\times r$ matrices with $B$-bit entries is $O(r^2B\log B + r^\omega B \log\log B)$, where $\omega < 3$ is the exponent of matrix multiplication.  We have $r=O(\log B)$, so this is $O(r^2B\log B)=O(r^2N\log N)$,
which we may substitute for \cite[Lem.\,3.1]{HS16} in the proof of \cite[Thm\,3.3]{HS16}. The cost of replacing vector-matrix multiplications with matrix multiplications as we transition from one subtree to the next is asymptotically negligible: we may reduce modulo $m\coloneqq \prod_{k=0}^{N-1}m_k$ throughout and perform $O(2^\kappa)=O(\log^2 N)$ matrix multiplications with $O(N)$-bit entries, each involving $O(r^2N\log N)$ bit operations.
\end{proof}

\begin{algorithm}\label{algorithm: avgpoly}
Given $f \in \Z[x_0,x_1,x_2]_4$ with $\Delta_4^*(f)\ne 0$ and a positive integer $N$, compute the Cartier--Manin matrices $A_p$ of the reductions of the smooth plane quartic $X\colon f(x_0,x_1,x_2)=0$ modulo primes $p\le N$ of good reduction for $X$ as follows:

\begin{enumerate}
\setlength{\itemsep}{3pt}
\item Use the \textsc{RemainderForest} algorithm to compute $C_p=\prod_{p-1>j\ge 0} M(j)\bmod p$ for primes $p\le N$ with $p\nmid D$ using the matrices $M_i\coloneqq M(-2-i)\in \Z^{16\times 16}$ and moduli $m_i\coloneqq i+2$ when $i+2$ is a prime $p\nmid D$ and with $m_i\coloneqq 1$ otherwise, for $0\le i < N-1$.  The matrices~$M_i$ and moduli $m_i$ should be dynamically computed as needed.
\item For each $C_p$ computed in (1) apply Algorithm~\ref{algorithm:cartier manin} with input $f\bmod p$ and $C_p$ to compute $A_p$.
This step should be interleaved with step (1), computing the relevant~$A_p$ in batches as the \textsc{RemainderForest}  algorithm completes each subtree.
\item For $p\le N$ of good reduction dividing $D$ compute $A_p$ via Remark~\ref{rem:badp}.
\end{enumerate}
\end{algorithm}

Note that for primes $p\le N$ that do not divide $D$ we have
\begin{align}\notag
P_{p-2} &= \prod_{i=0}^{p-2} M_i\bmod m_{p-2} = \prod_{i=0}^{p-2}M(-2-i)\bmod p\\
        &\equiv \prod_{i=0}^{p-2}M(p-2-i)\equiv \prod_{p-1>j\ge 0}M(j) \equiv C_p\bmod p,
\end{align}
thus step (1) of Algorithm~\ref{algorithm: avgpoly} computes exactly the matrices $C_p$ that are needed in step (2).

\begin{remark}\label{rem:divfactor}
Lemma \ref{lemma:compression} and Corollary~\ref{corollary:calT^s} imply that each integer matrix product $M_iM_{i+1}$ is divisible by $\lambda_6$. In our implementation of Algorithm~\ref{algorithm: avgpoly} we precompute $\lambda_6$ and remove it from each matrix product computed during the \textsc{RemainderForest} computation in step~(1). This changes the output $P_{p-2}\bmod p$ by a factor of $\lambda_6^{p-2}$, and we divide once more by $\lambda_6$ to obtain the desired matrix $C_p$, since $\lambda_6^{p-1}\equiv 1\bmod p$ (note that $\lambda_6\!\mid\! D$  so $p\nmid \lambda_6$).  This does not change the complexity of the algorithm, but it reduces the sizes of the matrix coefficients in every layer of the product tree above the leaves by roughly a factor of 2, which yields a significant constant factor speedup (more than a factor of 2 in our tests). \end{remark}

\begin{remark}\label{rem:uncompressedavgpoly}
As in Remark~\ref{rem:uncompressedmodp}, we may also consider an uncompressed version of Algorithm~\ref{algorithm: avgpoly} that instead computes $28\times 28$ matrices $U_p\bmod p$ and uses Remark~\ref{rem:uncompressedCM} to compute the Cartier--Manin matrices $A_p$.  In this uncompressed version we are not able to apply the optimization noted in Remark~\ref{rem:divfactor}.
\end{remark}

\begin{remark}\label{rem:ellcover}
Algorithms~\ref{algorithm:mod p} and~\ref{algorithm: avgpoly} can be modified to more efficiently handle smooth plane quartics of the form $f(x_0,x_1,x_2)=x_0^4+h(x_1,x_2)x_0^2+g(x_1,x_2)$. In this case $f^{p-1}_v=0$ whenever $v_0$ is odd, and for $p>2$ this implies that the Cartier--Manin matrix $A_p\in \F_p^{3\times 3}$ has at most five nonzero entries: the four corners and the center.  The center corresponds to the $1\times 1$ Cartier--Manin matrix of the genus 1 curve $x_0^2=h(x_1,x_2)^2-4g(x_1,x_2)$ which can be computed via \cite{HS16} using $4\times 4$ matrices.  Restricting the domain and codomain of
\[
\tau_{w^{(1)}+(2t+1)(e_0-e_1),0,1}^{\mathcal E}\circ \tau_{w^{(1)}+2t(e_0-e_1),0,1}^{\mathcal E}
\]
to the subspaces spanned by monomials with even degree in $x_0$ yields a matrix $M\in R[t]^{16\times 16}$. One finds that $M$ can be compressed via a coordinate projection to a $10\times 10 $ matrix $M'$, and we compute $W_p:=\prod_{\frac{p-3}{2}\ge k\ge 0}M(k)\bmod p$ as the product of $M(\frac{p-3}{2})$ and the zero extension of $\prod_{\frac{p-3}{2}> k \ge 0}M'(t)\bmod p$. The matrix $W_p$ can then be zero extended to $U_p\in \F_p^{28\times 28}$ and used to compute the four corner entries of $A_p$ via Remark~\ref{rem:uncompressedmodp}. 
\end{remark}

\begin{theorem}\label{thm:avgpoly}
Algorithm~\ref{algorithm: avgpoly} runs in $O(N\log^3\!N)$ time using $O(N)$ space.
\end{theorem}
\begin{proof}
Theorem~\ref{thm:forest} implies that the complexity of step (1) is within the desired bounds.  Step (2) calls Algorithm~\ref{algorithm:cartier manin} $O(N/\log N)$ times, which takes $O(N\log N\log\log N)$ time using $O(\log N)$ space.  The complexity of step (3) is asymptotically negligible, since $D$ is fixed as a function of $N$, and the theorem follows.
\end{proof}

To help assess the benefits of our new recurrences, we also implemented an algorithm that uses the recurrences derived in \cite{Har15} to compute the Cartier--Manin matrix $A_p$ of a smooth plane quartic $X:f(x_0,x_1,x_2)=0$ (or its reduction modulo $p$ when $R=\Z$).  If one applies \cite[Thm.\,4.1]{Har15} with $n=2$, $d=4$, $s=1$, $h=(d-1)(n+1)+1=10$, $k_0=p-1$, and $w=v+z$ with $z=(0,0,6)\in D_{h-d}$, one obtains a matrix $Q\in R[k ,l]^{66\times 66}$ that can be used to compute $\restr{f^{p-1}}{D(pv+z,10)}$ for any $v\in D_4$ via
\begin{equation}\label{eq:Qmat}
\restr{f^{p-1}}{D(pv+z,10)} = \frac{1}{d^{p-1}(p-1)!}Q(p-1,p-2)Q(p-1,p-3)\cdots Q(p-1,0)\restr{g^{p-1}}{D(pv+z,10)},
\end{equation}
where $g(x_0,x_1,x_2)=x_0^4+x_1^4+x_2^4$.  The algorithm in \cite{Har15} uses the matrix $Q$ to compute a matrix $M_s$ which is then used to compute the matrix $A_{F^s}^{ar}$ that appears in the trace formula \cite[Thm.\,3.1]{Har15}, but the Cartier--Manin matrix $A_p$ can be computed directly from \eqref{eq:Qmat}, and it suffices to compute the product $M(p-2)M(p-3)\cdots M(0)\bmod p$, where $M(j)\coloneqq Q(-1,j)$; the algorithm in \cite{Har15} works modulo $p^2$ when $s=1$, but that is not necessary here.  This product does not depend on $v\in D_4$, so it suffices to compute a single matrix product and then apply \eqref{eq:Qmat} using $v=(1,1,2), (2,1,1), (1,2,1)$; this yields three vectors in $\F_p^{66}$, each of which contains three entries that correspond to a column of $A_p$.

Having reduced the problem to computing $\prod_{p-1>j\ge 0}M(j)\bmod p$ we immediately obtain algorithms to compute $A_p$ with the complexities given in Theorem \ref{thm:modp} for $R=\F_p$, and for $R=\Z$ we obtain an average polynomial-time algorithm with the complexities given in Theorem \ref{thm:avgpoly} using a remainder forest.  The difference in the size of the matrices (66 versus 28 or 16) only impacts the constant factors, which we consider in the next section.

\begin{remark}\label{rem:haropt}
There is an additional optimization that we exploit in our implementation of the average polynomial-time algorithm based on \cite[Thm.\,4.1]{Har15}.  In the remainder forest algorithm, rather than computing the $66\times 66$ matrix $P_k=M_0\cdots M_k \bmod m_k$ we instead compute the $3\times 66$ matrix $P_k=V_0M_0\cdots M_k\bmod m_k$, where $V_0$ is a $3\times 66$ matrix with entries in $\{0,1\}$ and zeros in all but one entry of each row.  This optimization is possible because we only need 3 rows of the matrix product to compute $A_p$.  This optimization is not applicable in the context of Algorithm~\ref{algorithm: avgpoly} because we need to invert the reduced matrix products in order to compute the middle column of $A_p$ via Algorithm~\ref{algorithm:cartier manin}.
\end{remark}

A demonstration version of the $\widetilde O(p)$ and average polynomial-time versions of all three approaches (compressed, uncompressed, and the algorithm based on \cite[Thm.\,4.1]{Har15}) written in the SageMath computer algebra system \cite{sage} is available at \cite{CHS22}.  The optimized C implementation whose practical performance is analyzed in the next section will be part of the next release of the open source \texttt{smalljac} software library \cite{KS08}.

\section{Performance comparisons}\label{sec:timings}

In this section we compare the practical performance of our new algorithms to each other, and to existing implementations, both for computing the Cartier--Manin matrix of a smooth plane quartic over $\F_p$ (see Table~\ref{table:Fptimings}), and for computing the Cartier--Manin matrices of the reductions of a smooth plane quartic over $\Q$ at good primes $p\le N$ for some bound $N$.  Table~\ref{table:Qtimings1} compares the new average polynomial-time algorithms to each other and Table~\ref{table:Qtimings2} compares them to average polynomial-time algorithms for other types of genus 3 curves.

We first consider $\widetilde O(p)$ and $\widetilde O(p^{1/2})$ implementations of the compressed and uncompressed versions of Algorithm~\ref{algorithm:mod p} (denoted Algorithm~\ref{algorithm:mod p}c and Algorithm~\ref{algorithm:mod p}u below) as well as $\widetilde O(p)$ and $\widetilde O(p^{1/2})$ implementations of the approach based on \cite[Thm.\,4.1]{Har15} described at the end of the previous section (denoted \cite{Har15} (optimized) below).  We compared the performance of these six algorithms to each other, and to the following existing algorithms:

\begin{itemize}
\setlength{\itemsep}{8pt}
\item In \cite{Cos15} Costa gives an $\widetilde{O}(p)$-time $p$-adic algorithm for computing the matrix of Frobenius to a specified $p$-adic precision, which can be used to compute the Cartier--Manin matrix of a smooth plane quartic.  This algorithm is available at \cite{Cos15a}.
\item The \texttt{smalljac} software library \cite{KS08} includes a na\"ive point-counting algorithm for plane projective curves $X\colon f(x_0,x_1,x_2)=0$ that computes
\begin{equation}\label{eq:naive}
\qquad\ \#X(\F_p) = 0^{f(1,0,0)}\! + \#\{t\in \F_p\!:\!f(t,0,1)=0\} +\! \sum_{a\in \F_p} \#\{t\in \F_p:f(t,1,a)=0\}
\end{equation}
via the identity $\#\{t\in \F_p:g(t)=0\}=\deg\gcd(g(t),t^p-t)$ (valid for $g\ne 0$), in $O(p\log^2\!p\log\log p)$ time using $O(\log p)$ space.
\item For smooth plane curves the \texttt{RationalPoints} function in Magma \cite{Magma} uses an $O(p\log^2\!p\log\log p)$-time algorithm to enumerate rational points over $\F_p$.
\end{itemize}

The last two algorithms only compute $\#X(\F_p)$, they do not compute the Cartier--Manin matrix $A_p$, which provides additional information about $X$, including the reduction of its zeta function modulo $p$ and the $p$-rank of its Jacobian.  Magma includes an implementation of Tuitman's algorithm \cite{Tui16} that computes the entire zeta function in $\widetilde O(p)$ time, but the constant factors make it more than 100 times slower than the three $\widetilde O(p)$ algorithms listed above in the ranges we tested, so we chose not to include it in our comparison.

We ran each of these 9 algorithm on smooth plane quartics defined by dense polynomials $f\in \F_p[x_0,x_1,x_2]_4$, taking $p$ to be the first prime larger than $2^n$ for $n=10,11,\ldots,30$.  The running times for each algorithm can be found in  Table~\ref{table:Fptimings}, in which the complexity bounds in the column headings ignore $O(\log \log p)$ factors.

Each of the three $\widetilde O(p^{1/2})$ algorithms is substantially faster than the existing approaches, as one would expect given the asymptotic advantage.  For $p\approx 2^{30}$ Algorithm~\ref{algorithm:mod p}c appears to be faster than Algorithm~\ref{algorithm:mod p}u by factor of about $3$, which in turn appears to be faster than \cite{Har15} (optimized) by a factor of almost~$8$.  The factor of $3\approx (28/16)^2$ is as expected, while the factor of $8 > 5.6\approx (66/28)^2$ is larger than one might expect; this is likely due to the fact that $p$ is not large enough for the $O(r^\omega p^{1/2}\log p \log\log p)$  term in the complexity bound from \cite{BGS07} to become completely negligible.  All three implementations use the \texttt{smalljac} library \cite{KS08}, which includes an implementation of the algorithm in \cite{BGS07} built on the \texttt{zn\_poly} library \cite{Har10}, which is used for fast cache-friendly multiplication in $\F_p[x]$.

The relative performance of the $\widetilde O(p)$ implementations of Algorithm~\ref{algorithm:mod p} is perhaps more surprising: Algorithm~\ref{algorithm:mod p}u outperforms Algorithm~\ref{algorithm:mod p}c by a wide margin.  This is explained by the fact that in our $\widetilde O(p)$ implementation of Algorithm~\ref{algorithm:mod p}u we exploit the shape of the $28\times 28$ matrices $M(t)$ defined in Remark~\ref{rem:uncompressedmodp}: as can be seen from \eqref{eq:phivi}, it has only  $7\cdot 22 + 21 = 165 < 256=16^2$ nonzero entries.  As noted in Remark~\ref{rem:lineartime}, in our $\widetilde O(p)$ implementation we iteratively compute matrix-vector products, which lets us exploit the sparsity of the uncompressed $M(t)$ (the compressed matrices are not sparse).  Additionally, the uncompressed $M(t)$ have degree $1$ rather than $2$, which provides a further speedup.
  
\begin{table}
\vspace{-20pt}
\small
\setlength{\tabcolsep}{3pt}
\begin{tabular}{ l r r r r r r r r r r }
&\multicolumn{7}{c}{Cartier--Manin matrix} & \multicolumn{2}{c}{\!\!\!point counting}\\
\cmidrule(r){2-8}
\cmidrule(r){9-10}
&\multicolumn{2}{c}{Algorithm~\ref{algorithm:mod p}c}&\multicolumn{2}{c}{Algorithm~\ref{algorithm:mod p}u}& \multicolumn{2}{c}{\cite{Har15}\! (optimized)} & \multicolumn{1}{c}{\cite{Cos15}}& \multicolumn{1}{c}{\footnotesize\texttt{smalljac}\ } &\multicolumn{1}{c}{\texttt{magma}}\\
\cmidrule(l{3pt}r{3pt}){2-3}
\cmidrule(l{3pt}r{3pt}){4-5}
\cmidrule(l{3pt}r{3pt}){6-7}
\cmidrule(l{3pt}r{3pt}){8-8}
\cmidrule(l{1pt}r{1pt}){9-9}
\cmidrule(l{3pt}r{3pt}){10-10}
$\quad\ p$ & $p^{\nicefrac{1}{2}}\!\log^2\! p$ & $p\log p$ & $p^{\nicefrac{1}{2}}\!\log^2\! p$ & $p\log p$ & $p^{\nicefrac{1}{2}}\!\log^2\! p$ & $p\log p$ & $p\log p$ & $p\log^2\! p$ & $p\log^2\! p$\\
\toprule
$2^{10}+7$  & 0.003 & 0.001 & 0.002 & 0.000 & 0.022 & 0.001 & 0.014 & 0.000 & 0.000 \\
$2^{11}+5$  & 0.003 & 0.001 & 0.003 & 0.000 & 0.029 & 0.003 & 0.017 & 0.001 & 0.010 \\
$2^{12}+3$  & 0.004 & 0.002 & 0.004 & 0.000 & 0.041 & 0.006 & 0.023 & 0.001 & 0.020 \\
$2^{13}+17$ & 0.004 & 0.004 & 0.006 & 0.001 & 0.056 & 0.011 & 0.035 & 0.002 & 0.040 \\
$2^{14}+27$ & 0.005 & 0.009 & 0.008 & 0.002 & 0.081 & 0.023 & 0.058 & 0.004 & 0.070 \\
$2^{15}+3$  & 0.006 & 0.017 & 0.012 & 0.003 & 0.113 & 0.047 & 0.112 & 0.008 & 0.140 \\
$2^{16}+1$  & 0.008 & 0.033 & 0.018 & 0.006 & 0.175 & 0.089 & 0.192 & 0.023 & 0.300 \\
$2^{17}+29$ & 0.011 & 0.066 & 0.028 & 0.012 & 0.255 & 0.184 & 0.372 & 0.039 & 0.620\\
$2^{18}+3$  & 0.017 & 0.130 & 0.047 & 0.024 & 0.402 & 0.368 & 0.718 & 0.078 & 1.23\phantom{0}\\
$2^{19}+21$ & 0.025 & 0.263 & 0.072 & 0.047 & 0.598 & 0.735 & 1.43\phantom{0} & 0.158 & 2.62\phantom{0}\\
$2^{20}+7$  & 0.039 & 0.527 & 0.119 & 0.092 & 0.956 & 1.41\phantom{0} & 2.84\phantom{0} & 0.324 & 5.50\phantom{0} \\
$2^{21}+17$ & 0.060 & 1.05\phantom{0} & 0.186 & 0.188 & 1.47\phantom{0} & 2.84\phantom{0} &  5.65\phantom{0} &  0.740 & 11.4\phantom{00}\\
$2^{22}+15$ & 0.100 & 2.11\phantom{0} & 0.318 & 0.370 & 2.41\phantom{0} & 5.65\phantom{0} &  11.3\phantom{00} &  1.47\phantom{0} & 23.9\phantom{00}\\
$2^{23}+9$  & 0.154 & 4.15\phantom{0} & 0.488 & 0.736 & 3.69\phantom{0} & 11.8\phantom{00} &  22.6\phantom{00} &  2.93\phantom{0} & 48.3\phantom{00}\\
$2^{24}+43$ & 0.269 & 8.43\phantom{0} & 0.858 & 1.46\phantom{0} & 6.26\phantom{0} & 23.4\phantom{00} &  44.9\phantom{00} &  6.44\phantom{0} & 99.3\phantom{00}\\
$2^{25}+35$ & 0.421 & 16.6\phantom{00} & 1.35\phantom{0} & 2.93\phantom{0} & 9.73\phantom{0} & 45.2\phantom{00} & 89.9\phantom{00} & 13.6\phantom{00} & 201\phantom{.000}\\
$2^{26}+15$ & 0.735 & 33.7\phantom{00} & 2.36\phantom{0} & 5.83\phantom{0} & 16.8\phantom{00} & 90.4\phantom{00} & 180\phantom{.000} & 26.9\phantom{00} & 723\phantom{.000}\\
$2^{27}+29$ & 1.16\phantom{0} & 66.4\phantom{00} & 3.68\phantom{0} & 11.7\phantom{00} & 27.4\phantom{00} & 188\phantom{.000} & 360\phantom{.000} & 54.5\phantom{00} & 1530\phantom{.000}\\
$2^{28}+3$  & 1.95\phantom{0} & 135\phantom{.000} & 6.14\phantom{0} & 23.4\phantom{00} & 44.5\phantom{00} & 361\phantom{.000} & 719\phantom{.000} & 114\phantom{.000} & 3080\phantom{.000}\\
$2^{29}+11$ & 2.90\phantom{0} & 265\phantom{.000} & 9.04\phantom{0} & 46.7\phantom{00} & 68.5\phantom{00} & 750\phantom{.000}& 1440\phantom{.000} & 230\phantom{.000} & 6430\phantom{.000}\\
$2^{30}+3$  & 4.89\phantom{0} & 539\phantom{.000} & 15.1\phantom{00} & 93.1\phantom{00} & 119\phantom{.000} & \!\!1480\phantom{.000}& 3130\phantom{.000} & 465\phantom{.000} & 13600\phantom{.000}\\
\bottomrule
\end{tabular}
\smallskip
\caption{\small Algorithms for smooth plane quartics over $\F_p$. Times in 5.2GHz Intel i9-12900K core-seconds. Complexities ignore $O(\log\log p)$ factors.  The point counting computations only determine the trace of the Cartier--Manin matrix.}\label{table:Fptimings}
\medskip

\begin{tabular}{lrrrrrrrrrrr}
&\multicolumn{3}{c}{Algorithm~\ref{algorithm: avgpoly}c}&$\quad$&\multicolumn{3}{c}{Algorithm~\ref{algorithm: avgpoly}u} & $\medspace$& \multicolumn{3}{c}{\cite{Har15} (optimized)}\\
\cmidrule(lr){2-4}
\cmidrule(lr){6-8}
\cmidrule(lr){10-12}
$\ N\quad$ & seconds &\hspace{9pt} ms/$p$ & GB && seconds & ms/$p$ & $\quad$GB && seconds & ms/$p$ & GB\\\midrule
$2^{10}$ & 0.060 & 0.355 & 0.042 && 0.151 & 0.903 & 0.033 && 0.092 & 0.550 & 0.034\\
$2^{11}$ & 0.135 & 0.444 & 0.043 && 0.395 & 1.30\phantom{0} & 0.035 && 0.219 & 0.719 & 0.034\\
$2^{12}$ & 0.280 & 0.500 & 0.044 && 1.12\phantom{0} & 2.01\phantom{0} & 0.035 && 0.592 & 1.06\phantom{0} & 0.034\\
$2^{13}$ & 0.648 & 0.633 & 0.047 && 3.60\phantom{0} & 3.51\phantom{0} & 0.036 && 1.84\phantom{0} & 1.80\phantom{0} & 0.035\\
$2^{14}$ & 1.47\phantom{0} & 0.774 & 0.053 && 7.00\phantom{0} & 3.69\phantom{0} & 0.077 && 6.66\phantom{0} & 3.34\phantom{0} & 0.035\\
$2^{15}$ & 3.62\phantom{0} & 1.03\phantom{0} & 0.067 && 15.9\phantom{00} & 4.54\phantom{0} & 0.123 && 24.2\phantom{00} & 6.89\phantom{0} & 0.037\\
$2^{16}$ & 8.08\phantom{0} & 1.24\phantom{0} & 0.088 && 36.9\phantom{00} & 5.65\phantom{0} & 0.217 && 74.4\phantom{00} & 11.4\phantom{00} & 0.040\\
$2^{17}$ & 19.2\phantom{00} & 1.57\phantom{0} & 0.131 && 85.2\phantom{00} & 6.96\phantom{0} & 0.410 && 252\phantom{.000} & 20.5\phantom{00} & 0.071\\
$2^{18}$ & 44.8\phantom{00} & 1.95\phantom{0} & 0.223 && 192\phantom{.000} & 8.37\phantom{0} & 0.805 && 676\phantom{.000}& 29.4\phantom{00} & 0.910\\
$2^{19}$ & 106\phantom{.000} & 2.44\phantom{0} & 0.413 && 437\phantom{.000} & 10.1\phantom{00} & 1.63\phantom{0} && 1680\phantom{.000}& 38.6\phantom{00} & 2.38\phantom{0}\\
$2^{20}$ & 241\phantom{.000} & 2.94\phantom{0} & 0.790 && 991\phantom{.000} & 12.1\phantom{00} & 3.29\phantom{0} && 4100\phantom{.000}& 50.0\phantom{00} & 4.91\phantom{0}\\
$2^{21}$ & 543\phantom{.000} & 3.49\phantom{0} & 1.57\phantom{0} && 2230\phantom{.000} & 14.3\phantom{00} & 6.73\phantom{0} && 10800\phantom{.000}& 69.3\phantom{00} & 10.1\phantom{00}\\
$2^{22}$ & 1260\phantom{.000} & 4.26\phantom{0} & 3.20\phantom{0} && 5040\phantom{.000} & 17.0\phantom{00} & 13.8\phantom{00} && 29900\phantom{.000}& 101\phantom{.000} & 20.9\phantom{00}\\
$2^{23}$ & 2950\phantom{.000} & 5.23\phantom{0} & 6.57\phantom{0} && 11400\phantom{.000} & 20.3\phantom{00} & 28.4\phantom{00} && 88200\phantom{.000}& 156\phantom{.000} & 43.2\phantom{00}\\
\bottomrule
\end{tabular}
\smallskip

\caption{\small Average polynomial-time algorithms for  smooth plane quartics over $\Q$ with small coefficients. Times in 5.2GHz Intel i9-12900K core-seconds.}\label{table:Qtimings1}
\end{table}

We also analyzed the performance of the three average polynomial-time algorithms introduced in this paper: the compressed and uncompressed versions of Algorithm~\ref{algorithm: avgpoly} and the algorithm based on \cite[Thm.\,4.1]{Har15}.  Table~\ref{table:Qtimings1} lists the total time and space, and average time per prime, to compute the Cartier--Manin matrices of the reductions modulo~$p$ of a fixed smooth plane quartic curve over $\Q$ for good primes $p\le N=2^n$ for $n=10,11,\ldots, 23$.  We used a dense polynomial $f\in\Z[x_0,x_1,x_2]_4$ with small (single digit) coefficients as input to all three algorithms.
The parameter $\kappa$ that determines the number $2^\kappa$ of trees in the remainder forest was chosen to optimize the running time; for $N=2^{18},\ldots, 2^{23}$ this led us to use $\kappa=6$ for both versions of Algorithm~\ref{algorithm: avgpoly} and $\kappa=7$ for the algorithm based on \cite[Thm.\,4.1]{Har15}, which is close to the asymptotic value $\kappa = \lfloor 2\log_2\log_2\! N\rfloor$ used in Theorem~\ref{thm:forest}.

\begin{remark}
For the algorithm based on \cite[Thm.\,4.1]{Har15}, at small values of $N$ the optimal value of $\kappa$ is actually $\log_2 N$, meaning that each ``tree'' in the forest consists of a single matrix.  This choice of $\kappa$ leads to an $\widetilde O(N^2)$ time complexity but is advantageous for small values of~$N$ because it allows the algorithm to avoid full matrix multiplications via Remark~\ref{rem:haropt}.  This explains the rapid growth in the running times for this algorithm for $N\le 2^{17}$.
\end{remark}

In addition to $\kappa$, the memory used by our algorithms is influenced by the matrix dimensions and the size of the matrix coefficients.
To get a better understanding of these parameters, we analyzed the computation of a single product tree in the middle of a remainder forest with $N=2^{24}$ and $\kappa=6$ for all three algorithms.  The results are shown in Table~\ref{table:sizes}, in which one can see the growth in the size of the matrix coefficients at each level of the product tree in the ``KB/entry'' columns, the total size of all the matrices in each level in the ``MB'' columns, and the total time per level.  The decrease in the total size of the matrices in the first few layers of the product tree for Algorithm~\ref{algorithm: avgpoly}c is explained by Remark~\ref{rem:divfactor}.

\begin{remark}
In our implementation we use the algorithm for integer matrix multiplication described in \cite{HvdH18}.  As explained in the proof of Theorem~\ref{thm:avgpoly}, this algorithm computes the product of $r\times r$ matrices with $b$-bit entries in time $O(r^2 b\log b + r^\omega b\log\log b)$, provided that $\log r=O(\log b)$.  This becomes $O(r^2b\log b)$ when $b$ is large relative to $r$, as in the context of Theorem~\ref{thm:forest} where we have $r=O(\log B)$, and in Theorem~\ref{thm:avgpoly} where $r=O(1)$.  But for the small values of~$b$ that arise in the lower levels of the product tree the constant factors make this approach less efficient than na\"ive matrix multiplication, so we use the algorithm of \cite{HvdH18} only once it becomes faster to do so.  These crossover points are indicated by thin horizontal lines in Table~\ref{table:sizes}.
Given that $r$ is fixed in all the algorithms we consider, we made no attempt to achieve the optimal value of $\omega$ in our implementation; doing so might have improved the relative performance of the algorithm with $r=66$ in the range we tested.
\end{remark}

\begin{table}
\small
\begin{tabular}{crcccrcccrcc}
&\multicolumn{3}{p{4cm}}{\centering Algorithm~\ref{algorithm: avgpoly}c\\$\ (r=16)$}&&\multicolumn{3}{p{4cm}}{\centering Algorithm~\ref{algorithm: avgpoly}u\\$\ (r=28)$} && \multicolumn{3}{p{4cm}}{\centering\cite{Har15} (optimized)\\$\ (r=66)$}\\
\cmidrule(lr){2-4}
\cmidrule(lr){6-8}
\cmidrule(lr){10-12}
products & KB/entry & MB & seconds && KB/entry & MB & seconds && KB/entry & MB & seconds\\\toprule
$2^{17}$ & 0.014 & 457 & 2.91 && 0.005 & 469 & 6.62 && 0.003 & 1890 & 87.2\\
$2^{16}$ & 0.029 & 470 & 2.95 && 0.015 & 776 & 6.21 && 0.009 & 2508 & 70.7\\
$2^{15}$ & 0.055 & 449 & 2.28 && 0.039 & 989 & 7.37 && 0.019 & 2624 & 53.5\\
$2^{14}$ & 0.103 & 420 & 2.44 && 0.079 & 996 & 7.07 && 0.038 & 2679 & 36.3\\
$2^{13}$ & 0.198 & 406 & 2.62 && 0.159 & 999 & 8.68 && 0.078 & 2708 & 31.8\\[-3pt]\cmidrule(l{24pt}r{8pt}){2-4}\\[-16pt]
$2^{12}$ & 0.389 & 399 & 3.58 && 0.319 & 1001 & 13.2\phantom{0} && 0.156 & 2723 & 46.0\\[-3pt]\cmidrule(l{24pt}r{8pt}){6-8}\\[-16pt]
$2^{11}$ & 0.772 & 395 & 3.71 && 0.639 & 1002 & 14.4\phantom{0} && 0.313 & 2730 & 73.6\\[-3pt]\cmidrule(l{24pt}r{8pt}){10-12}\\[-16pt]
$2^{10}$ & 1.54\phantom{9} & 393 & 3.44 && 1.28\phantom{0} & 1003 & 13.6\phantom{0} && 0.628 & 2734 & 79.6\\
$2^{9}$ & 3.07\phantom{1} & 392 & 3.39 && 2.56\phantom{0} & 1003 & 14.0\phantom{0} && 1.26\phantom{0} & 2736 & 77.6\\
$2^{8}$ & 6.13\phantom{0} & 392 & 3.43 && 5.12\phantom{0} & 1003 & 14.1\phantom{0} && 2.51\phantom{0} & 2737 & 76.8\\
$2^{7}$ & 12.2\phantom{00} & 392 & 3.51 && 10.2\phantom{00} & 1003 & 14.4\phantom{0} && 5.03\phantom{0} & 2737 & 76.5\\
$2^{6}$ & 24.5\phantom{00} & 392 & 3.81 && 20.5\phantom{00} & 1003 & 15.0\phantom{0} && 10.1\phantom{00} & 2737 & 77.9\\
$2^{5}$ & 49.0\phantom{00} & 392 & 3.90 && 40.9\phantom{00} & 1003 & 15.2\phantom{0} && 20.1\phantom{00} & 2738 & 80.0\\
$2^{4}$ & 97.9\phantom{00} & 392 & 4.05 && 81.9\phantom{00} & 1003 & 15.5\phantom{0} && 40.2\phantom{00} & 2738 & 80.8\\
$2^{3}$ & 196\phantom{.000} & 392 & 4.18 && 164\phantom{.000} & 1003 & 16.0\phantom{0} && 80.4\phantom{00} & 2738 & 82.0\\
$2^{2}$ & 392\phantom{.000} & 392 & 4.37 && 328\phantom{.000} & 1003 & 16.5\phantom{0} && 161\phantom{.000} & 2738 & 84.1\\
$2$ & 783\phantom{.000} & 392 & 4.52 && 655\phantom{.000} & 1003 & 17.1\phantom{0} && 322\phantom{.000} & 2738 & 85.7\\
$1$ & 1570\phantom{.000} & 392 & 5.80 && 1310\phantom{.000} & 1003 & 21.0\phantom{0} && 644\phantom{.000} & 2738 & 96.4\\
\bottomrule
\end{tabular}
\smallskip

\caption{Computation of a product tree in the middle of a remainder forest with $N=2^{24}$ and $\kappa=6$ involving the product of $N/2^{\kappa}=2^{18}$ $r\times r$ matrices.  The ``MB'' columns list the total size of the products in megabytes.  Horizontal lines indicate matrix multiplication algorithm crossovers. Times in 5.2GHz Intel i9-12900K core-seconds.}\label{table:sizes}
\end{table}

In Table~\ref{table:sizes} one can see that the matrix coefficient sizes roughly double in each level while the number of matrix products is cut in half, and the total size of the products in each level is essentially constant in the top half of each tree.
Asymptotically, the time to build each layer of the product tree is quasilinear in the total size, so for sufficiently large $N/2^\kappa$ one would expect the relative running times of the three algorithms in the top half of the tree to approach the ratios of these total sizes, which are roughly $1:2.6:7.0$ for the algorithms with $r=16$, $28$, $66$, respectively. The ratios of the actual times to build these trees for $N=2^{24}$ are approximately $1:3.6:20.0$, a discrepancy that is likely explained by lower order complexity terms involving $r^\omega$ and the greater frequency of cache misses for larger total bit sizes.

\begin{remark}\label{rem:buildonly}
Table~\ref{table:sizes} only captures the cost of building a product tree in the remainder forest, which is less than half the total running time (for the time-optimal value of $\kappa$).  The other phases of the algorithm (transferring information between product trees and computing remainders down the trees) involve computations on matrices that one can assume have been reduced modulo $m$, where $m$ is either the product of all remaining moduli, or the product of the moduli in some subtree.  The values of $m$ will be the same in all three algorithms, so one would asymptotically expect the relative costs of these phases to converge to the relative ratios of $2r^2$ for $r=16,28$ and $3r+r^2$ for $r=66$ (via Remark~\ref{rem:haropt}), which are $1:3.1:8.9$.
\end{remark}

\begin{remark}
As in Table~\ref{table:Qtimings1}, the data in Table~\ref{table:sizes} reflects a curve with small coefficients, which is the case we expect to most often arise in practice (as in \cite{Sut19}, for example).
To assess the performance of our algorithms on curves with larger coefficients we also tested random curves with 10 and 100 digit coefficients with $N=2^{24}$ using $\kappa=8$ and $\kappa=10$.  As in Table~\ref{table:sizes}, the total size of the matrix products at each level stabilizes in the top half of the product tree, as do the relative running times.  For 10-digit coefficients the relative size ratios are  $1:2.8:2.7$ and the time ratios are $1:3.5:6.0$ (for the algorithms with $r=16,28,66$, respectively), and for 100-digit coefficients the relative size ratios are $1:2.7:1.8$ and the time ratios are $1:2.4:2.7$ (as noted above, these ratios are relevant only to the build phase).
\end{remark}

Finally, we compared the performance of Algorithm~\ref{algorithm: avgpoly}c to average polynomial-time algorithms that are applicable to various types of genus 3 curves over $\Q$, including:

\begin{itemize}
\setlength{\itemsep}{6pt}
\item The algorithm in \cite{HMS16} for computing Cartier--Manin matrices of reductions of a geometrically hyperelliptic curve of genus 3 defined over $\Q$ with a model of the form $g(x,y,z)=0, w^2=f(x,y,z)$, where $g$ is a pointless conic and $\deg f=4$.
\item The algorithm in \cite{HS16} for computing Cartier--Manin matrices of reductions of a hyperelliptic curve over $\Q$, applied to a genus 3 curve $y^2=f(x)$ with $\deg f=8$, which is a 2-cover of $\P^1$.
\item The algorithm in \cite{Sut20} for computing the Cartier--Manin matrices of reductions of superelliptic curves $y^m=f(x)$ over $\Q$ applied to genus 3 curves of the form $y^3=f(x)$ and $y^4=f(x)$ with $\deg f = 4$ (the case $y^3=f(x)$ is a Picard curve).
\item The algorithm for smooth plane quartics of the form $x^4+h(y,z)x^2=f(y,z)$ (these are degree 2 covers of genus 1 curves) described in Remark~\ref{rem:ellcover}.
\end{itemize}

The results appear in Table~\ref{table:Qtimings2}, which reflects curves defined by dense polynomials with random single digit coefficients.  All of these implementations use the \textsc{RemainderForest} algorithm and the same libraries for multiplying polynomials and matrices over $\F_p$ and $\Z$, based on \cite{Har10} and \cite{HvdH18}. None of these computations required more than 64GB memory, but the computations for smooth plane quartics were the most memory intensive.

\begin{table}[h!]
\begin{tabular}{lrrrrrr}
\multicolumn{1}{p{1cm}}{\centering \quad\\\!\!\!\!\!\!$N$} & \multicolumn{1}{p{1.6cm}}{\centering plane\\ quartic}  & \multicolumn{1}{p{2.3cm}}{\centering geometrically \\ hyperelliptic} & \multicolumn{1}{p{2.3cm}}{\centering rationally \\ hyperelliptic} & \multicolumn{1}{p{2.4cm}}{\centering 2-cover of a\\ genus 1 curve} & \multicolumn{1}{p{1.5cm}}{\centering 3-cover\\of $\P^1$} & \multicolumn{1}{p{1.5cm}}{\centering $4$-cover \\ of $\P^1$}\\
\toprule
$2^{10}$ & 0.058 & 0.053 & 0.007 & 0.021 & 0.006 & 0.006 \\
$2^{11}$ & 0.158 & 0.069 & 0.008 & 0.035 & 0.007 & 0.007 \\
$2^{12}$ & 0.281 & 0.126 & 0.011 & 0.070 & 0.008 & 0.008 \\
$2^{13}$ & 0.638 & 0.294 & 0.022 & 0.139 & 0.013 & 0.012 \\
$2^{14}$ & 1.49\phantom{0} & 0.724 & 0.065 & 0.326 & 0.030 & 0.028\\
$2^{15}$ & 3.43\phantom{0} & 2.12\phantom{0} & 0.222 & 0.742 & 0.086 & 0.089\\
$2^{16}$ & 8.00\phantom{0} & 5.42\phantom{0} & 0.829 & 1.77\phantom{0} & 0.333 & 0.285\\
$2^{17}$ & 19.1\phantom{00} & 12.4\phantom{00} & 3.25\phantom{0} & 4.24\phantom{0} & 0.882 & 0.760\\
$2^{18}$ & 44.6\phantom{00} & 29.6\phantom{00} & 10.0\phantom{00} & 10.1\phantom{00} & 2.38\phantom{0} & 2.15\phantom{0}\\
$2^{19}$ &  105\phantom{.000} & 69.5\phantom{00} & 24.4\phantom{00} & 24.2\phantom{00} & 6.67\phantom{0} & 5.48\phantom{0}\\
$2^{20}$ &  241\phantom{.000} & 168\phantom{.000} & 55.6\phantom{00} & 57.2\phantom{00} & 15.3\phantom{00} & 12.2\phantom{00}\\
$2^{21}$ &  543\phantom{.000} & 388\phantom{.000} & 133\phantom{.000} & 133\phantom{.000} & 36.1\phantom{00} & 29.6\phantom{00}\\
$2^{22}$ & 1260\phantom{.000} & 921\phantom{.000} & 320\phantom{.000} & 315\phantom{.000} & 87.6\phantom{00} & 72.0\phantom{00}\\
$2^{23}$ & 2950\phantom{.000} & 2160\phantom{.000} & 746\phantom{.000} & 748\phantom{.000} & 214\phantom{.000} & 173\phantom{.000}\\
$2^{24}$ & 6840\phantom{.000} & 4860\phantom{.000} & 1760\phantom{.000} & 1750\phantom{.000} & 514\phantom{.000} & 410\phantom{.000}\\
$2^{25}$ & 15600\phantom{.000} & 11200\phantom{.000} & 4120\phantom{.000} & 4050\phantom{.000} & 1220\phantom{.000} & 975\phantom{.000}\\
$2^{26}$ & 35600\phantom{.000} & 26000\phantom{.000} & 9560\phantom{.000} & 9370\phantom{.000} & 2880\phantom{.000} & 2350\phantom{.000}\\
\bottomrule
\end{tabular}
\smallskip

\caption{Average polynomial-time algorithms for various genus~3 curves over~$\Q$ with small coefficients.  Times in 5.2GHz Intel i9-12900K core-seconds.}\label{table:Qtimings2}\end{table}


\begin{thebibliography}{HvdH19}

\bibitem[AH19]{AH19}
Jeffrey D. Achter and Everett W. Howe, \href{https://doi.org/10.1090/conm/722/14534}{\textit{Hasse--Witt and Cartier--Manin matrices: A warning and a request}}, Arithmetic Geometry: Computations and Applications, Contemporary Mathematics \textbf{722} (2019), 1--18, American Mathematical Society. (\mr{3896846}, \arxiv{1710.10726v5})

\bibitem[AH01]{AH01}
Leonard M. Adleman and Ming-Deh Huang, \href{https://doi.org/10.1007/3-540-61581-4_36}{\textit{Counting points on curves and abelian varieties over finite fields}}, International Algorithmic Number Theory Symposium (ANTS I), LNCS \textbf{1122} (1996), 1--16, Springer. (\mr{1446493})

\bibitem[Bat93]{Bat93}
Victor V. Batyrev, \href{https://doi.org/10.1215/S0012-7094-93-06917-7}{\textit{Variations of the mixed Hodge structure of affine hypersurfaces in algebraic tori}}, Duke Math J. \textbf{69} (1993). (\mr{1203231})

\bibitem[Magma]{Magma}
Wieb Bosma, John Cannon, and Catherine Playoust, \href{https://doi.org/10.1006/jsco.1996.0125}{\textit{The Magma algebra system. I. The user language}}, J. Symbolic Comput. \textbf{24} (1997), 235--265. (\mr{1484478})

\bibitem[BGS07]{BGS07}
Alan Bostan, Pierrick Gaudry and \'Eric Schost. \href{https://doi.org/10.1137/S0097539704443793}{\textit{Linear recurrences with polynomial coefficients and application to integer factorization and {C}artier--{M}anin operator}}, SIAM J. Comput. \textbf{36} (2007), 1777--1806. (\mr{2299425}, \halinria{00103401})

\bibitem[CV09]{CV09}
Wouter Castryck and John Voight, \href{https://doi.org/10.2140/ant.2009.3.255}{\textit{On nondegeneracy of curves}}, Algebra Number Theory \textbf{3} (2009), 255--281. (\mr{2525551}, \arxiv{0802.0420})

\bibitem[CV10]{CV10}
Wouter Castryck and John Voight, \href{https://doi.org/10.1090/conm/521/10270}{\textit{Nondegenerate curves of low genus over small finite fields}}, Arithmetic, Geometry, Cryptography, and Coding Theory 2009, Contemporary Mathematics \textbf{521} (2010), 21--28, American Mathematical Society (\mr{2744030}, \arxiv{0907.2060})

\bibitem[Cos15]{Cos15}
Edgar Costa, \href{https://www.proquest.com/docview/1711150592}{\textit{Effective computations of Hasse-Weil zeta functions}}, Ph.D. thesis, New York University, 2015 (\mr{3419250})

\bibitem[Cos15a]{Cos15a}
Edgar Costa, \href{https://github.com/edgarcosta/pycontrolledreduction/}{\texttt{PycontrolledReduction}}, GitHub repository, \url{https://github.com/edgarcosta/controlledreduction} (retrieved March 2021)

\bibitem[CHS22]{CHS22}
Edgar Costa, David Harvey, and Andrew V. Sutherland, \href{https://cocalc.com/AndrewVSutherland/SPQPointCounting/ToyImplementation}{\texttt{SPQPointcounting}}, Jupyter~notebook, \url{https://cocalc.com/AndrewVSutherland/SPQPointCounting/ToyImplementation} (2022).

\bibitem[FKS21]{FKS21}
Francesc Fit\'e, Kiran S. Kedlaya, and Andrew V. Sutherland, \href{https://doi.org/10.1090/conm/770}{\textit{Sato--Tate groups of abelian threefolds: a preview of the classification}}, in Arithmetic Geometry, Cryptography, and Coding Theory, Contemp. Math. \textbf{770} (2021), 103--129. (\mr{4280389}, \arxiv{1911.02071})

\bibitem[FKS22]{FKS22}
Francesc Fit\'e, Kiran S. Kedlaya, and Andrew V. Sutherland, \href{https://doi.org/10.48550/arXiv.2106.13759}{\textit{Sato--Tate groups of abelian threefolds}}, preprint. \arxiv{2106.13759}

\bibitem[FOR08]{FOR08}
St\'ephane Flon, Roger Oyono, and Christophe Ritzenthaler, \href{https://doi.org/10.1142/9789812793430_0001}{\textit{Fast addition on non-hyperelliptic genus 3 curves}}, in Algebraic Geometry and its Applications, Ser. Number Theory Appl. \textbf{5} (2008) 1--28, World Sci. Publ. (\mr{2484046}, \iacr{2004/118})

\bibitem[GG13]{GG13}
Joachim von zur Gathen and J\"urgen Gerhard, \href{https://doi.org/10.1017/CBO9781139856065}{\textit{Modern computer algebra}}, third edition, Cambridge University Press, 2013. (\mr{3087522})

\bibitem[GKZ94]{GKZ94}
Israel M. Gelfand, Mikhail M. Kapranov, Andrei V. Zelevinsky, \href{https://doi.org/10.1007/978-0-8176-4771-1}{\textit{Discriminants, resultants, and multidimensional determinants}}, Birkh\"auser, 1994. (\mr{2394437})

\bibitem[Har10]{Har10}
David Harvey, \href{https://doi.org/10.1016/j.tcs.2009.03.014}{\textit{A cache-friendly truncated FFT}}, Theoret. Comput. Sci. \textbf{410} (2009), 2649--2658. (\mr{2531107}, \arxiv{0810.3203})

\bibitem[Har15]{Har15}
David Harvey, \href{https://doi.org/10.1112/plms/pdv056}{\textit{Computing zeta functions of arithmetic schemes}}, Proc. Lond. Math. Soc. \textbf{111} (2015), 1379--1401. (\mr{3447797} \arxiv{1402.3439})

\bibitem[HvdH18]{HvdH18}
David Harvey and Joris van der Hoeven, \href{https://doi.org/10.1016/j.jsc.2017.11.001}{\textit{On the complexity of integer multiplication}}, J. Symbolic Comput. \textbf{89} (2018). (\mr{3804803}, \hal{01071191})

\bibitem[HvdH21]{HvdH21}
David Harvey and Joris van der Hoeven, \href{https://doi.org/10.4007/annals.2021.193.2.4}{\textit{Integer multiplication in time $O(n\log n)$}}, Annals of Math. \textbf{193} (2021), 563--617. (\mr{4224716}, \hal{02070778})

\bibitem[HMS16]{HMS16}
David Harvey, Maike Massierer, and Andrew V. Sutherland, \href{https://doi.org/10.1112/S1461157016000383}{\textit{Computing $L$-series of geometrically hyperelliptic curves of genus three}}, in Algorithmic Number Theory 12th International Symposium (ANTS XII), LMS J. Comput. Math. \textbf{19A} (2016), 220--234. (\mr{3540957}, \arxiv{1605.04708})

\bibitem[HS14]{HS14}
David Harvey and Andrew V. Sutherland, \href{https://doi.org/10.1112/S1461157014000187}{\textit{Computing {H}asse--{W}itt matrices of hyperelliptic curves in average polynomial time}}, Algorithmic Number Theory 11th International Symposium (ANTS XI), LMS J. Comput. Math. \textbf{17A} (2014), 257--273. (\mr{3240808}, \arxiv{1402.3246})

\bibitem[HS16]{HS16}
David Harvey and Andrew V. Sutherland, \href{https://doi.org/10.1090/conm/663/13352}{\textit{Computing {H}asse--{W}itt matrices of hyperelliptic curves in average polynomial time, II}}, in Frobenius Distributions: Lang--Trotter and Sato--Tate Conjectures, Contemp. Math. \textbf{663} (2016), 127--147, American Mathematical Society. (\mr{3502941} \arxiv{1410.5222})

\bibitem[Katz73]{Katz73}
Nicholas M. Katz, \href{https://doi.org/10.1007/BFb0060518}{\textit{Une formule de congruence pour la function $\zeta$}}, in Groups de Monodromie en G\'eom\'etrie Alg\'ebrique, Lecture Notes in Mathematics \textbf{340} (1973), 401--438, Springer. (\mr{0354657})

\bibitem[KS08]{KS08}
Kiran S. Kedlaya and Andrew V. Sutherland, \href{https://doi.org/10.1007/978-3-540-79456-1_21}{\textit{Computing L-series of hyperelliptic curves}}, Algorithmic Number Theory 8th International Symposium (ANTS VIII), Lecture Notes in Computer Science \textbf{487} (2009), 119--162, Springer. (\mr{2467855}, \arxiv{0801.2778})

\bibitem[Ma1916]{Mac1916}
F.S. Macaulay, \href{https://doi.org/10.3792/chmm/1263317740}{\textit{The algebraic theory of modular systems}}, Cornell Hist. Math Monographs, 1916. (\mr{1281612})

\bibitem[Man65]{Man65}
Ju. I. Manin, \href{https://doi.org/10.1090/trans2/045}{\textit{The Hasse--Witt matrix of an algebraic curve}}, in Fifteen Papers on Algebra, Amer. Math. Soc. Transl. \textbf{45} (1965) 245--264, translated by J.W.S. Cassels. (\mr{0124324})

\bibitem[Pil90]{Pil90}
Jonathan Pila, \href{https://doi.org/10.2307/2008445}{\textit{Frobenius maps of abelian varieties and fining roots of unity in finite fields}}, Math. Comp. \textbf{55} (1990), 745--763 (\mr{1035941})

\bibitem[Sage]{sage}
The Sage Developers, \href{https://www.sagemath.org}{\texttt{SageMath}}, the Sage Mathematics Software System Version 9.4, available at \url{https://www.sagemath.org}, 2021.

\bibitem[Sti09]{Sti09}
Henning Stichtenoth, \href{https://doi.org/10.1007/978-3-540-76878-4}{\textit{Algebraic function fields and codes}}, Graduate Texts in Mathematics \textbf{254}, Springer, 2009. (\mr{2464941})

\bibitem[SV87]{SV87}
Karl-Otto St\"ohr and Jos\'e Felipe Voloch, \href{https://doi.org/10.1515/crll.1987.377.49}{\textit{A formula for the {C}artier operator on plane algebraic curves}}, J. Reine Angew. Math. \textbf{377} (1987), 49--64. (\mr{0887399})

\bibitem[Sch85]{Sch85}
Ren\'e Schoof, \href{https://doi.org/10.1090/S0025-5718-1985-0777280-6}{\textit{Elliptic curves over finite fields and the computation of square roots mod $p$}}, Math. Comp. \textbf{44} (1985), 483--494. (\mr{777280})

\bibitem[Sut07]{Sut07}
Andrew V. Sutherland, \href{https://dspace.mit.edu/handle/1721.1/38881}{\textit{Order computations in generic groups}}, PhD Thesis, Massachusetts Institute of Technology, 2007. (\mr{2717420})

\bibitem[Sut09]{Sut09}
Andrew V. Sutherland, \href{https://doi.org/10.1090/S0025-5718-08-02143-1}{\textit{A generic approach to searching for Jacobians}}, Math. Comp. \textbf{78} (2009), 485--507. (\mr{2448717}, \arxiv{0708.3168})

\bibitem[Sut19]{Sut19}
Andrew V. Sutherland, \href{https://msp.org/obs/2019/2-1/p27.xhtml}{\textit{A database of nonhyperelliptic curves over $\Q$}}, Thirteenth Algorithmic Number Theory Symposium (ANTS XIII), Open Book Series \textbf{2} (2019), 443--459. (\mr{3952027}, \arxiv{1806.06289})

\bibitem[Sut20]{Sut20}
Andrew V. Sutherland, \href{https://doi.org/10.2140/obs.2020.4.403}{\textit{Counting points on superelliptic curves in average polynomial time}}, Fourteenth Algorithmic Number Theory Symposium (ANTS XIV), Open Book Series \textbf{4} (2020), 403--422. (\mr{4235126}, \arxiv{2004.10189})

\bibitem[Tui16]{Tui16}
Jan Tuitman, \href{https://doi.org/10.1090/mcom/2996}{\textit{Counting points on curves using a map to $\P^1$}}, Math. Comp. \textbf{85} (2016), 961--981. (\mr{3434890}, \arxiv{1402.6758})

\end{thebibliography}
\end{document}